\documentclass[12pt]{article}
\usepackage{amssymb,amsmath, amsthm,latexsym, verbatim, amscd}
\usepackage[hypertex]{hyperref}
\usepackage{here}
  \newtheorem{theorem}{Theorem}[section] %
  \newtheorem{proposition}[theorem]{Proposition} %
  \newtheorem{lemma}[theorem]{Lemma} %
  
  \newtheorem{corollary}[theorem]{Corollary} %
  
  \newtheorem{definition-theorem}[theorem]{Definition-Theorem}
  \newtheorem{example}[theorem]{Example} %
  \newtheorem{prob}[theorem]{Problem}
  \newtheorem{question}[theorem]{Question}
  \newtheorem{observation}[theorem]{Observation}
  \newtheorem{remark}[theorem]{Remark} %
\def\Disc#1{{\operatorname{Disc}({#1})}}

\def \Hom{\operatorname{Hom}}
\def \Hsum#1{{{\underset{#1}{{\sum}^{\oplus}}}}}

\newcommand \itm[1]{\newline\noindent{\rm{#1}}\enspace}

\def\Ad{\operatorname{Ad}}

\def\app{{+ +}}
\def\apm{{+ -}}
\def\amp{{- +}}
\newcommand\pip[3]{\pi^{{#1},{#2}}_{+, {#3}}}
\newcommand\pim[3]{\pi^{{#1},{#2}}_{-, {#3}}}
\newcommand\pia[4]{\pi^{{#2},{#3}}_{{#1}, {#4}}}
\newcommand\der[1]{\frac{\partial}{\partial{#1}}}
\newcommand \set[2]{\{{#1}:{#2}\}}
\def\rarrowsim{\smash{\mathop{\,\rightarrow\,}\limits
  ^{\lower1.5pt\hbox{$\scriptstyle\sim$}}}}

\newcommand\V[4]{V_{{#1}, {#4}}^{({#2}, {#3})}}

\newcommand \gk{$(\frak g, K)$}
\newcommand{\invHom}[3]{\operatorname{Hom}_{#1}(#2,#3)}

\numberwithin{equation}{section}
\numberwithin{equation}{section}
\numberwithin{table}{section}

\begin{document}
\title{Branching laws of unitary representations 
 associated to minimal elliptic orbits
 for indefinite orthogonal group $O(p,q)$
} 
\author{Toshiyuki KOBAYASHI
\\
Graduate School of Mathematical Sciences
 and Kavli IPMU (WPI)
\\
The University of Tokyo}
\date{} %
\maketitle %

{MSC 2010: Primary  22E46; 
          Secondary 
                    22E45, 
53D50, 
58J42, 
53C50. 
}
\begin{abstract}
We give a complete description of the discrete spectra
 in the branching law $\Pi|_{G'}$
 with respect to the pair $(G,G')=(O(p,q), O(p',q') \times O(p'',q''))$
 for irreducible unitary representations $\Pi$ of $G$
 that are \lq\lq{geometric quantization}\rq\rq\ 
 of minimal elliptic coadjoint orbits.  
We also construct explicitly all holographic operators
 and prove a Parseval-type formula.  
\end{abstract}

\section{Introduction and main results}
\label{sec:Intro}
In this article,
 we determine the discrete spectra
 of the restriction $\Pi|_{G'}$
 of an irreducible unitary representation of $G$
 to a subgroup $G'$, 
 where 
\begin{enumerate}
\item[$\bullet$]
$\Pi$ is \lq\lq{attached to}\rq\rq\
 a minimal elliptic coadjoint orbit
 (Section \ref{sec:elliptic}), 
\item[$\bullet$]
$(G,G')=(O(p,q), O(p',q') \times O(p'',q''))$
 with $p=p'+p''$ and $q=q'+q''$.  
\end{enumerate}
We denote by $\widehat {G'}$ the set of equivalence classes
 of irreducible unitary  representations of $G'$
 ({\it{unitary dual}}).  
In Theorem \ref{thm:2002}
 we prove a {\it{multiplicity-free theorem}} asserting
\[
  \dim_{\mathbb{C}} \Hom_{G'}(\pi,\Pi|_{G'}) \le 1
  \quad
  \text{for all $\pi \in \widehat {G'}$}, 
\]
 and give a complete description of the {\it{discrete spectra}}
 for the branching:
\[
  \operatorname{Disc}(\Pi|_{G'})
  :=
  \{ \pi \in \widehat {G'}: \Hom_{G'}(\pi,\Pi|_{G'}) \ne \{0\}
   \}, 
\]
where $\Hom_{G'}(\,,\,)$ denotes the space
 of {\it{continuous}} $G'$-homomorphisms.

The irreducible unitary representations $\Pi$ 
 in consideration are of various aspects such as
\begin{enumerate}
\item[$\bullet$]
they are \lq\lq{geometric quantization}\rq\rq\
 of indefinite K{\"a}hler manifolds
  (Section \ref{subsec:Dolbeault});
\item[$\bullet$]
they are \lq\lq{discrete series representations}\rq\rq\
 for pseudo-Riemannian space forms
 (Section \ref{subsec:Xpq}), 
  \cite{xfar, xstri};
\item[$\bullet$]
they are \lq\lq{unitarization}\rq\rq\
 of the Zuckerman derived functor modules
 that are cohomological induction from a maximal $\theta$-stable
 parabolic subalgebra
 ${\mathfrak {q}}$
 (Section \ref{subsec:Aq}), 
 \cite{xvr, xvz}.  
\end{enumerate}
The representations $\Pi$ of $G=O(p,q)$
 are parametrized by $\varepsilon \in \{\pm\}$
 and $\lambda \in A_{\varepsilon}(p,q)$, 
 see Definition-Theorem \ref{def:pilmd}, 
 and will be denoted by 
$
   \pi_{\varepsilon, \lambda}^{p,q}.  
$

Our first main result gives a description of the discrete part 
 ({\it{cf}}. Section \ref{subsec:Pidisc})
 of the restriction $\Pi|_{G'}$.  
Without loss of generality, 
 we assume $\varepsilon=+$.  
\begin{theorem}
\label{thm:2002}
For $\lambda \in A_+(p,q)$,
 we set $\Pi=\pi_{+,\lambda}^{p,q}$, 
 the irreducible unitary representation
 of $G=O(p,q)$, 
 as in Definition-Theorem \ref{def:pilmd}.  
Then the discrete part
 of the restriction $\Pi|_{G'}$ is a multiplicity-free
 direct sum of irreducible unitary representations
 of the subgroup $G'=O(p', q') \times O(p'',q'')$
 as follows:

\begin{equation}
\label{eqn:2.1.1}
\bigoplus_{(\delta,\varepsilon) \in \{\amp, \app, \apm\}}
\Hsum{(\lambda', \lambda'') \in \Lambda_{\delta \varepsilon}(\lambda)} 
     \pi_{\delta, \lambda'}^{p',q'} \boxtimes 
     \pi_{\varepsilon, \lambda''}^{p'',q''}
\quad
     \text{{\rm{(Hilbert direct sum)}}.}
\end{equation}
\end{theorem}

Here the parameter set $\Lambda_{\delta,\varepsilon}(\lambda)$ is defined 
 for $\lambda \in A_+(p,q)$ by 
\begin{align*}
\Lambda_\amp(\lambda) 
  &:=\set{(\lambda', \lambda'') \in A_-(p',q') \times A_+(p'',q'')}{
           \lambda'' - \lambda - \lambda' - 1 \in 2 \Bbb N}, 
\\
\Lambda_\app(\lambda) 
  &:=\set{(\lambda', \lambda'') \in A_+(p',q') \times A_+(p'',q'')}{
         \lambda - \lambda' - \lambda'' - 1 \in 2 \Bbb N}, 
\\
\Lambda_\apm(\lambda) 
  &:=\set{(\lambda', \lambda'') \in A_+(p',q') \times A_-(p'',q'')}{
          \lambda' - \lambda'' - \lambda - 1 \in 2 \Bbb N}.  
\end{align*}
We note that $\Lambda_{++}(\lambda)$ is a finite set, 
 whereas $\Lambda_{\apm}(\lambda)$ (also $\Lambda_{\amp}(\lambda)$)
 is an infinite set 
 unless it is empty.

Our proof is geometric and constructive.  
It is outlined as follows.  
First, 
 we divide the pseudo-Riemannian space form 
 $G/H =O(p,q)/O(p-1,q)$
 into three regions (up to conull set)
 according to orbit types
 labeled by $\amp$, $\app$, $\apm$
 of the subgroup $G'$. 
Second, 
 we introduce $G'$-intertwining operators
 ({\it{holographic operators}}) from each irreducible summand
 of \eqref{eqn:2.1.1}
 to the original representation $\pi_{+,\lambda}^{p,q}$
 by realizing these representations
 in the space of eigenfunctions of the Laplacian 
 on pseudo-Riemannian space forms 
 (Theorem \ref{thm:holographic}).  
The final step is to prove the exhaustion
 of \eqref{eqn:2.1.1}, 
 which is carried out 
 by a careful estimate 
 of the boundary behaviours
 of solutions 
 that \lq\lq{holographic operators}\rq\rq\
 must satisfy (Section \ref{sec:5}).  

Here is an example of Theorem \ref{thm:2002}
 when $(p'', q'')=(1,0)$ and $(0,1)$.  

\begin{example}
\label{ex:GP}
Suppose $p \ge 2$ and $q\ge 1$.  
Let $\Pi := \pi_{+,\lambda}^{p,q} \in \widehat G$
 for $\lambda \in A_+(p,q)$.  
\begin{enumerate}
\item[{\rm{(1)}}]
{\rm{(\cite{xk:1})}}\enspace
If $(p'', q'')=(0,1)$, 
 then $\Lambda_{\amp}(\lambda)=\Lambda_{\app}(\lambda)=\emptyset$
 and 
\[
   \Pi|_{G'} 
   =
   \Hsum{n \in {\mathbb{N}}} 
   \pi_{+,\lambda+n+\frac 1 2}^{p,q-1}
   \boxtimes 
   (\operatorname{sgn})^{n}, 
\]
where $\operatorname{sgn}$ stands for the nontrivial character
 of $O(1) \simeq O(1,0)$.  
\item[{\rm{(2)}}]
If $(p'', q'')=(1,0)$, 
 then $\Lambda_{\amp}(\lambda)=\Lambda_{\apm}(\lambda)=\emptyset$.  
Moreover,
 $\Hom_{G'}(\pi,\Pi|_{G'}) \ne \{0\}$
 if and only if $\pi \in \widehat {G'}$ is of the form
\[
   \pi=
   \pi_{+,\lambda-n-\frac 1 2}^{p-1,q}
   \boxtimes
   (\operatorname{sgn})^{n}
\qquad
\text{for some $0 \le n < \lambda-\frac 1 2$}.  
\]
\end{enumerate}
\end{example}
In the general case where $p',p'',q',q'' \ge 2$
 and $\lambda > 2$, 
 all the three parameter sets $\Lambda_{\amp}(\lambda)$, 
$\Lambda_{\app}(\lambda)$, 
 and $\Lambda_{\apm}(\lambda)$
 are nonempty 
 (Section \ref{sec:comments}).

As a corollary of Theorem \ref{thm:2002} and its proof, 
 we find a necessary and sufficient condition
 on the quadruple $(p',p'',q',q'')$
 for the restriction $\Pi|_{G'}$
 to have the following properties:
\begin{enumerate}
\item[$\bullet$]
$\Pi|_{G'}$ is discretely decomposable
 (Theorem \ref{thm:discdeco}), 
\item[$\bullet$]
the discrete part \eqref{eqn:2.1.1} is at most a finite sum
 (Theorem \ref{thm:191419}), 
\item[$\bullet$]
 $\Pi|_{G'}$ contains only continuous spectrum 
 (Theorem \ref{thm:conti}).  
\end{enumerate}

Our results can be also applied
 to the existence problem
 of symmetry breaking operators
 between {\it{smooth representations}} of $G$ and its subgroup $G'$.  
Let $\Pi^{\infty}$ be the Fr{\'e}chet space of smooth vectors
 of the unitary representation $\Pi$ of $G$, 
 and $\pi^{\infty}$ that of a unitary representation 
 $\pi$ of the subgroup $G'$.  

\begin{corollary}
\label{cor:SBO}
Let $\Pi=\pi_{+,\lambda}^{p,q} \in \widehat G$ 
 for $\lambda \in A_+(p,q)$
 and $\pi=\pi_{\delta,\lambda'}^{p',q'} \boxtimes \pi_{\varepsilon,\lambda''}^{p'',q''}\in \widehat{G'}$
 for some $(\delta, \varepsilon)=(-,+)$, $(+,+)$, or $(+,-)$.  
Then we have:
\begin{equation}
\label{eqn:SBO}
  \Hom_{G'}(\Pi^{\infty}|_{G'}, \pi^{\infty})\ne\{0\}
\quad
\text{if $(\lambda', \lambda'') \in \Lambda_{\delta,\varepsilon} (\lambda)$}.  
\end{equation}
\end{corollary}

The second main theorem in this article
 is a quantitative result:
 for every $(\lambda', \lambda'') \in \Lambda_{\delta,\varepsilon} (\lambda)$, 
 we construct explicitly
 in a geometric model of representations
 a holographic operator
 (an injective $G'$-intertwining operator)
\[
  T_{\delta\varepsilon,\lambda}^{\lambda',\lambda''}
  \colon
  \pi_{\delta,\lambda'}^{p',q'} \boxtimes \pi_{\varepsilon,\lambda''}^{p'',q''}
  \to
  \pi_{+,\lambda}^{p,q}, 
\]
 and find a closed formula
 of its operator norm
 (Theorem \ref{thm:holographic}).

\vskip 0.8pc
Branching laws in the same setting with specific choices
 of $p'$, $p''$, $q'$, $q''$
 have been studied over 25 years:
\begin{enumerate}
\item[$\bullet$]
When $(p'',q')=(0,0)$, 
 Theorem \ref{thm:2002} is nothing but the $K$-type formula, 
 and can be computed by a generalized Blattner formula
 of the Zuckerman derived functor modules
 \cite{xvr, xk92}, 
 see also Faraut \cite{xfar}, Howe--Tan \cite{xhowetan}.  

\item[$\bullet$]
When $p''=0$, 
 the restriction $\Pi|_{G'}$ is discretely decomposable
 (Theorem \ref{thm:discdeco}).  
In this case,
 Theorem \ref{thm:2002} gives the whole branching law
of the restriction $\Pi|_{G'}$, 
 which was determined in \cite[Thm.~3.3]{xk:1}.  
The special case $(p,q)=(3,3)$  with $(p'', q'')=(0,1)$ was also studied in \cite{xspeh}.  

\item[$\bullet$]
When $(q',q'')=(1,0)$ (hence $q=1$), 
 the branching law of $\Pi|_{G'}$ was obtained in \cite{xmo}.  
In this case, 
 $\Pi|_{G'}$ contains also continuous spectrum.  

\item[$\bullet$]
In the case $p''=q=1$, 
 an analogous result to \eqref{eqn:SBO} was studied 
  in \cite[Thms.~4.1 and 4.2]{xksbonvec}
 when $\Pi^{\infty}$ and $\pi^{\infty}$ are cohomologically induced
 representations from more general parabolic subalgebras.  

\item[$\bullet$]
If $(p'',q'')=(1,0)$ or $(0,1)$, 
 then $\Hom_{G'}(\Pi^{\infty}|_{G'}, \pi^{\infty})$ is at most
 of one-dimensional by the general result of Sun and Zhu \cite{xsunzhu}.  
In this case, 
 the discrete spectra \eqref{eqn:2.1.1} are stated
 in Example \ref{ex:GP}, 
 and some part of them
 have been obtained recently in {\O}rsted and Speh \cite{xso}
 by a different approach
 under the constraints
 that $b(\lambda) \ge 0$
 (see \eqref{eqn:b} for notation).  
\end{enumerate}

For general $p'$, $q'$, $p''$, $q''$, 
 the complete classification of discrete spectra
(Theorem \ref{thm:2002}),
 and the construction of all holographic operators
 with a Parseval-type theorem 
 (Theorems \ref{thm:holographic} and \ref{thm:4.2})
 were presented at the conference
 \lq\lq{Analyse harmonique sur les groupes de Lie
 et les espaces sym\'etriques}\rq\rq\
 en l'honneur de Jacques Faraut held in Nancy-Strasbourg 
 in June, 2005, 
 however, the manuscript \cite{xkmin}
 has not been published.

Because of growing interest in branching problems
 for reductive groups in recent years, 
 I come to think
 that the results and the methods here might be of some help 
 for further perspectives
 such as a possible generalization 
 of the Gross--Prasad conjecture
 for nontempered representations
 ({\it{e.g.}} \cite{xgp, xksbonvec, xso})
 as well as analytic representation theory.

\vskip 1pc
{\bf{$\langle$Acknowledgements$\rangle$}}\enspace
The author was partially supported
 by Grant-in-Aid for Scientific Research (A) (18H03669), 
Japan Society for the Promotion of Science.

\vskip 1pc
{\bf{Notation:}}\enspace
${\mathbb{N}}=\{0,1,2,\dots\}$ and 
${\mathbb{N}}_+=\{1,2,\dots\}$.

\section{Irreducible unitary representations
 attached to minimal elliptic orbits}
\label{sec:elliptic}
In this section,
 we discuss a certain family of irreducible unitary representations
 of $G=O(p,q)$,
 denoted by $\pi_{\varepsilon,\lambda}^{p,q}$ 
 with parameter $\varepsilon=\pm$
 and $\lambda \in A_{\varepsilon}(p,q)$
defined as below:
\begin{align}
\label{eqn:A+}
A_+(p,q):=&
\begin{cases}
\{\lambda \in {\mathbb{Z}}+\frac{p+q}2: \lambda >0\} 
&(p \ge 2,q\ge 1), 
\\
\{\lambda \in {\mathbb{Z}}+\frac{p}2: \lambda \ge \frac p 2-1\} 
&(p\ge 2,q=0), 
\\
\emptyset
&(p=1,q\ge 1)\,\, \text{ or }\,\, (p=0),  
\\
\{-\frac 1 2, \frac 1 2\} 
&(p=1,q=0).  
\end{cases}
\\
\label{eqn:A-}
A_-(p,q):=&A_+(q,p).  
\end{align}
The representations $\pi_{\varepsilon,\lambda}^{p,q}$ are a generalization 
 of the finite-dimensional representations
 of the compact group $O(p)$
 on the space ${\mathcal{H}}^m({\mathbb{R}}^p)$
 of spherical harmonics
 (see Remark \ref{rem:2.2} (1)).  
These unitary representations $\pi_{\varepsilon,\lambda}^{p,q}$
 have been treated from various aspects
 in scattered literatures
 (\cite{xfar, xhowetan, xk92, xk:1, opq2, xspeh, xso, xstri}).  
For the convenience of the reader,
 we summarize a number of realizations
 of the representations $\pi_{\varepsilon,\lambda}^{p,q}$
 when $\varepsilon=+$ 
 in Section \ref{subsec:pilmd}.  

Throughout this section, 
 we adopt the same notation as in \cite{opq2}.

\subsection{Summary: four realizations of $\pi_{\varepsilon,\lambda}^{p,q}$}
\label{subsec:pilmd}
We use the German lower case letter ${\mathfrak{g}}$, 
 ${\mathfrak{k}}$, $\cdots$, 
 to denote the Lie algebras
 of $G$, $K$, $\cdots$, 
 and write ${\mathfrak{Z}}({\mathfrak{g}})$
 for the center of the enveloping algebra
 of the complexified Lie algebra
 ${\mathfrak{g}}_{\mathbb{C}}
 ={\mathfrak{g}} \otimes_{\mathbb{R}} {\mathbb{C}}$.
For ${\mathfrak{g}}={\mathfrak{o}}(p,q)$, 
 we set
\begin{equation}
\label{eqn:rho}
\rho:=\frac 1 2 (p+q-2).  
\end{equation}

For $\lambda \in A_+(p,q)$, 
 we put
\begin{align}
b\equiv\, & b_+(\lambda,p,q):=\lambda-\frac p 2 + \frac q 2 + 1 \in {\mathbb{Z}}, 
\label{eqn:b}
\\
\label{eqn:e}
\delta \equiv\, & \delta_+ (\lambda,p,q)
:=(-1)^b.  
\end{align}

\begin{definition-theorem}
\label{def:pilmd}
Let $p \ge 2$ and $q \ge 0$.  
For any $\lambda \in A_+(p,q)$, 
 there exists a unique irreducible unitary representation
 of $G=O(p,q)$, 
 to be denoted by $\pi_{+, \lambda}^{p,q}$, 
 whose underlying \gk-module is given by
 one of (therefore, any of) the following \gk-modules
 that are isomorphic to each other:
\begin{enumerate}
\item[{\rm{(i)}}]
The Zuckerman derived functor module
 $A_{\mathfrak{q}}(\lambda-\rho)$ 
 (see Section \ref{subsec:Aq});
\item[{\rm{(ii)}}]
(geometric quantization of coadjoint orbits)\enspace
the underlying \gk-module of the Dolbeault cohomology 
 $H_{\overline \partial}^{p-2}({\mathcal{O}}_{\lambda}, {\mathcal{L}}_{\lambda+\rho})$
 (see Section \ref{subsec:Dolbeault});
\item[{\rm{(iii)}}]
the underlying \gk-module of the subrepresentation 
of the parabolic induction $I_{\delta}(\lambda+\rho)$
 with $K$-types $\Xi(K;b)$
 (see Section \ref{subsec:ps});
\item[{\rm{(iii)$'$}}]
the underlying \gk-module of the quotient 
 of the parabolic induction $I_{\delta}(-\lambda+\rho)$
 with $K$-types $\Xi(K;b)$;
\item[{\rm{(iv)}}]
the underlying \gk-module of the discrete series representation
 $L^2(X(p,q))_{\lambda}$
 (see Section \ref{subsec:Xpq})
 for the symmetric space
 $X(p,q)=O(p,q)/O(p-1,q)$.  
\end{enumerate}
The ${\mathfrak{Z}}({\mathfrak{g}})$-infinitesimal character 
 of $\pi_{+, \lambda}^{p,q}$ is given by
\begin{equation}
\label{eqn:Zginf}
  (\lambda,\frac{p+q}2-2,\frac{p+q}2-3,\cdots, \frac{p+q}2-[\frac{p+q}2])
\end{equation}
in the Harish-Chandra parametrization
 for the standard basis, 
 and the minimal $K$-type of $\pi_{+, \lambda}^{p,q}$
 is given by
\[
\begin{cases}
{\mathcal{H}}^b({\mathbb{R}}^p) \boxtimes {\bf{1}}
\quad
&\text{if $b \ge 0$}, 
\\
{\bf{1}} \boxtimes {\bf{1}}
\quad
&\text{if $b \le 0$}.  
\end{cases}
\]
\end{definition-theorem}

The proof of the equivalence is given in \cite[Thm.~3]{xk92}
 and \cite[Sect.~5.4]{opq2}, 
 see also references therein.  
Since these rich aspects
 of the representations $\pi_{\varepsilon, \lambda}^{p,q}$
 are the heart of our main results
 in both the proof and perspectives,
 we give a brief account
 on each of these aspects
 in Sections \ref{subsec:Aq}--\ref{subsec:Xpq} below.  

\begin{remark}
\label{rem:2.2}
\begin{enumerate}
\item[{\rm{(1)}}]
When $q=0$, 
 $\pi_{+, \lambda}^{p,0}$ is an irreducible finite-dimensional representation
 of the compact group $O(p,0) \simeq O(p)$ 
 on the space ${\mathcal{H}}^m({\mathbb{R}}^p)$
 of spherical harmonics of degree $m=\lambda-\frac p 2+1$.  
\item[{\rm{(2)}}]
The conditions {\rm{(iii)}} and {\rm{(iii)$'$}}
 in Definition-Theorem \ref{def:pilmd}
 make sense for $q >0$;
the other conditions for $q \ge 0$.  
\end{enumerate}
\end{remark}

For $(p,q)=(1,0)$, 
 $O(p,q) \simeq O(1)$.  
It is convenient to set
\[
  \text{
  $A_+(p,q)=\{\tfrac 1 2, -\tfrac 1 2\}\,\,$
 and $\,\,\pi_{+, \lambda}^{1,0}
  :=
  \begin{cases}
  {\bf{1}}\quad&\text{if $\lambda=-\tfrac 1 2$, }
\\
  {\operatorname{sgn}}\quad&\text{if $\lambda=\tfrac 1 2$.}
  \end{cases}$
}
\]
Via the isomorphism of Lie groups $O(p,q) \simeq O(q,p)$,
 we define an irreducible unitary representation
 $\pi_{-, \lambda}^{p,q}$
 for $\lambda \in A_-(p,q)$
 to be the one $\pi_{+, \lambda}^{q,p}$ of $O(q,p)$, 
 where we recall from \eqref{eqn:A-}
 that $A_-(p,q)=A_+(q,p)$.

By the $K$-type formula (see the condition (iii)
 in Definition-Theorem \ref{def:pilmd}
 and by the formula \eqref{eqn:Zginf}
 of the ${\mathfrak{Z}}({\mathfrak{g}})$-infinitesimal character, 
 the following proposition holds.  
\begin{proposition}
\label{prop:pilmd}
Irreducible unitary representations of $G=O(p,q)$ 
 in the following set are not isomorphic to each other:
\[
 \{\pi_{+, \lambda}^{p,q}
   : \lambda \in A_+(p,q)
 \}
  \cup 
 \{\pi_{-, \lambda}^{p,q}
   : \lambda \in A_-(p,q).  
 \}
\]
\end{proposition} 

\subsection{Zuckerman derived functor modules $A_{\mathfrak {q}}(\lambda)$}
\label{subsec:Aq}
Let $G=O(p,q)$, 
 and $\theta$ the Cartan involution
 corresponding to a maximal compact subgroup $K=O(p) \times O(q)$.  
We take a Cartan subalgebra ${\mathfrak{t}}$ of ${\mathfrak{k}}$, 
 and extend it to that of ${\mathfrak{g}}$, 
 to be denoted by ${\mathfrak{j}}$.  
Take the standard basis 
 $\{f_i: 1 \le i \le [\frac{p+q}2]\}$
 of ${\mathfrak{j}}_{\mathbb{C}}^{\ast}$
 such that the root system 
 $\Delta({\mathfrak{g}}_{\mathbb{C}}, {\mathfrak{j}}_{\mathbb{C}})$
 is given by
\[
  \{\pm f_i \pm f_j
    : 1 \le i < j \le [\frac{p+q}{2}]\}
\,\,
   (\cup \{\pm f_i: 1 \le i \le [\frac{p+q}{2}]\} 
\,\,\,
  (\text{$p+q$: odd})).  
\]
Let ${\mathfrak{q}}={\mathfrak{l}}_{\mathbb{C}} + {\mathfrak{u}}$
 be a $\theta$-stable parabolic subalgebra
 of ${\mathfrak{g}}_{\mathbb{C}}$
 with Levi part ${\mathfrak{l}}_{\mathbb{C}}$ containing ${\mathfrak{j}}_{\mathbb{C}}$
 and nilpotent radical ${\mathfrak{u}}$ defined by
\[
   \Delta({\mathfrak{u}}, {\mathfrak{j}}_{\mathbb{C}})
   =\{ f_1 \pm f_j
    : 2 \le j \le [\frac{p+q}{2}]\}
\,\,
   (\cup \{f_1\} 
\,\,\,
  (\text{$p+q$: odd})).  
\]
Then the normalizer $L$ of ${\mathfrak{u}}$ in $G$ is given by
\begin{equation}
\label{eqn:Levi}
   L \simeq SO(2) \times O(p-2,q).  
\end{equation}

For $\nu \in {\mathbb{Z}}$, 
 we write ${\mathbb{C}}_{\nu f_1}$
 for the one-dimensional representation of the Levi subgroup $L$
 by letting the second factor act trivially.  
The same letter ${\mathbb{C}}_{\nu f_1}$ is used 
to denote a character of the Lie algebra ${\mathfrak{l}}$
 for $\nu \in {\mathbb{C}}$.

Zuckerman introduced cohomological parabolic induction
 ${\mathcal{R}}_{\mathfrak{q}}^j$
 ($j \in {\mathbb{N}}$) 
 which is a covariant functor from the category 
 of $({\mathfrak{l}}, L \cap K)$-modules
 (or that of metaplectic $({\mathfrak{l}}, L \cap K)\tilde{}$-modules)
 to that of \gk-modules.

We note that ${\mathbb{C}}_{\lambda f_1}$ lifts 
 to the metaplectic $({\mathfrak{l}}, L \cap K)\tilde{}$-module
 if and only if ${\mathbb{C}}_{(\lambda+\rho) f_1}$ lifts to $L$, 
 namely,
 $\lambda \in {\mathbb{Z}}+\frac 1 2(p+q)$.  
In particular,
 for $\lambda \in A_+(p,q)$ ($\subset {\mathbb{Z}}+\frac 1 2 (p+q)$), 
 we obtain \gk-modules ${\mathcal{R}}_{\mathfrak{q}}^j({\mathbb{C}}_{\lambda f_1})$ for $j \in {\mathbb{N}}$, 
 which vanish except for $j =p-2$, 
 and the resulting \gk-module is 
\[
   {\mathcal{R}}_{\mathfrak{q}}^{p-2}({\mathbb{C}}_{\lambda f_1})
   \simeq
   A_{\mathfrak{q}}(\lambda-\rho).  
\]
Here we have adopted the convention and normalization
 in \cite[Def.~6.20]{xvr} for ${\mathcal{R}}_{\mathfrak{q}}^j$
 and in \cite{xvz} for $A_{\mathfrak{q}}(\cdot)$.  
This normalization means
 that $A_{\mathfrak{q}}(\nu)$ has nonzero \gk-cohomologies
 when $\nu=0$, 
 whereas ${\mathcal{R}}_{\mathfrak{q}}^j$ preserves
 the ${\mathfrak{Z}}({\mathfrak{l}})$- and ${\mathfrak{Z}}({\mathfrak{g}})$-infinitesimal characters
 in the Harish-Chandra parametrization modulo
 the Weyl groups $W_L$ and $W_G$.  

The general theory of the Zuckerman cohomological parabolic induction
 (see \cite{xvr} for instance)
 assures
 that the \gk-module 
 ${\mathcal{R}}_{\mathfrak{q}}^{p-2}({\mathbb{C}}_{\lambda f_1})$
 is nonzero and irreducible
 if $\lambda$ is in the \lq\lq{good range}\rq\rq\
 ({\it{i.e.}} if $\lambda > \frac 1 2 (p+q)-2$), 
 whereas the same condition may fail 
 if the parameter $\lambda$ wanders outside the \lq\lq{good range}\rq\rq.  
Although our parameter set $A_+(p,q)$ contains
 finitely many $\lambda$
 that are outside the good range, 
 the \gk-module 
 ${\mathcal{R}}_{\mathfrak{q}}^{p-2}({\mathbb{C}}_{\lambda f_1})$
 is nonzero and irreducible
 for all $\lambda \in A_+(p,q)$, 
 see \cite[Thm.~3]{xk92} applied to $r=1$
 with the notation therein.  

\subsection{Geometric quantization of elliptic orbits}
\label{subsec:Dolbeault}
Any coadjoint orbit of a Lie group carries 
 a natural symplectic structure.   
We shall see
 that the irreducible unitary representation $\pi_{+, \lambda}^{p,q}$ of $G$
 may be regarded as a \lq\lq{geometric quantization}\rq\rq\ 
 of the minimal elliptic coadjoint orbit
\[
  {\mathcal{O}}_{\nu}\equiv {\mathcal{O}}_{+, \nu} := \Ad^{\ast}(G)(\nu f_1)
  \,\,
  (\subset \sqrt{-1} {\mathfrak{g}}^{\ast}), 
\]
 where $\lambda = \nu-\rho$
 if we adopt the normalization
 of the parameter for \lq\lq{quantization}\rq\rq\
 as in \cite{xrons}, see below.

As a homogeneous space,
 ${\mathcal{O}}_{\nu}$ ($\nu \ne 0$) is identified with the homogeneous space
 $G/L$
 where $L$ is the subgroup defined in \eqref{eqn:Levi}.  
Since the same homogeneous space $G/L$ arises an open $G$-orbit 
of the complex flag variety $G_{\mathbb{C}}/Q$
 where $Q$ is the complex parabolic subgroup 
 with Lie algebra ${\mathfrak{q}}$
 (Section \ref{subsec:Aq})
 of the complexified Lie group $G_{\mathbb{C}}$, 
 it carries a $G$-invariant complex structure.  
Moreover, 
 it admits a $G$-invariant indefinite K{\"a}hler metric 
 such that its imaginary part yields the Kostant--Kirillov--Souriau symplectic form.

For $\nu \in {\mathbb{Z}}$, 
 we form a homogeneous line bundle 
 ${\mathcal{L}}_\nu:=G \times_L {\mathbb{C}}_{\nu f_1}$ over $G/L$.  
For instance, 
 the canonical bundle of $G/L$ is expressed as
 ${\mathcal{L}}_{2\rho} = {\mathcal{L}}_{p+q-2}$.  
For $\lambda \in {\mathbb{Z}} + \rho$ with $\lambda \ne 0$, 
 we take the Dolbeault cohomologies
 for the $G$-equivariant holomorphic line bundle
\[
   {\mathcal{L}}_{\lambda+\rho} \to {\mathcal{O}}_{\lambda} \simeq G/L, 
\]
 which carry a natural Fr{\'e}chet topology
 by the closed range theorem 
 of the $\overline \partial$-operator 
due to Schmid and Wong 
 \cite{xwong}, 
 and the Fr{\'e}chet $G$-module 
\[
   H_{\overline \partial}^{j}(G/L, {\mathcal{L}}_{\lambda+\rho})
\]
is a maximal globalization of the \gk-module
 ${\mathcal{R}}_{\mathfrak{q}}^{j}({\mathbb{C}}_{\lambda f_1})$.  
This shows the \gk-modules
 in (i) and (ii) in Theorem \ref{thm:2002}
 are isomorphic to each other.  
If $\lambda \in A_+(p,q)$, 
 then the Dolbeault cohomology for $j=p-2$ contains a Hilbert space
 on which $G$ acts as the unitary representation $\pi_{+,\lambda}^{p,q}$.

For $q \ge 2$, 
 we can consider similar family 
 of minimal elliptic coadjoint orbits 
${\mathcal{O}}_{-,\lambda} \simeq G/L_-$
 with $L_-:=O(p,q-2) \times SO(2)$
 by switching the role of $p$ and $q$, 
 and we obtain an irreducible unitary representations 
 $\pi_{-,\lambda}^{p,q}$ for $\lambda \in A_-(p,q)$
 ($=A_+(q,p)$).

The irreducible unitary representations $\pi_{\varepsilon,\lambda}^{p,q}$
 of $G$
 may be interpreted
 as geometric quantization
 of the coadjoint orbits ${\mathcal{O}}_{\varepsilon,\lambda}$, 
 and the Gelfand--Kirillov dimension is given by 
\[
   \operatorname{DIM} \pi_{\varepsilon, \lambda}^{p,q}
   =
   \frac 1 2 \dim {\mathcal{O}}_{\varepsilon, \lambda}
   =
   p+q-2
  \quad
  \text{for $\varepsilon=\pm$}.  
\]

\subsection{Degenerate principal series representations}
\label{subsec:ps}

The indefinite orthogonal group $G=O(p,q)$ has
 a maximal (real) parabolic subgroup 
 $P=M A N$, 
 unique up to conjugation,
 with Levi factor
\[
   M A \simeq GL (1,{\mathbb{R}}) \times O(p-1,q-1).  
\]
Any one-dimensional representation of the first factor $GL (1,{\mathbb{R}})$
 is parametrized by 
 $(\varepsilon,\nu) \in \{\pm \} \times {\mathbb{C}}$, 
 which extends to a character $\chi_{\varepsilon,\nu}$ of $M A$
 by letting the second factor trivial.  
We denote by $I_{\varepsilon}(\nu)$
 the $G$-module
 obtained as unnormalized parabolic induction
 $\operatorname{Ind}_P^G(\chi_{\varepsilon,\nu})$.  
Our parameter $\nu$ is chosen in a way
 that the trivial one-dimensional representation ${\bf{1}}$ of $G$
 occurs as the subrepresentation of $I_+(0)$, 
 and as the quotient of $I_+(2\rho)=I_+(p+q-2)$.

Geometrically,
 the real flag variety $G/P$ has a $G$-equivariant double covering
\begin{equation}
\label{eqn:SSGP}
  S^{p-1} \times S^{q-1} \simeq G/ P_+ \to G/P
\end{equation}
where $P_+ = (G L (1,{\mathbb{R}})_+ \times O(p-1,q-1))N$ is a normal subgroup 
 of $P$ of index two, 
 and the group $G$ acts conformally on $S^{p-1} \times S^{q-1}$
 endowed with the pseudo-Riemannian metric
 $g_{S^{p-1}} \oplus (-g_{S^{q-1}})$.

We recall that ${\mathcal{H}}^m({\mathbb{R}}^p)$ denotes
 the space of spherical harmonics
 of degree $m$.  
For $p=1$, 
 we consider only $m=0$ and $1$.  
The orthogonal group $O(p)$ acts irreducibly on ${\mathcal{H}}^m({\mathbb{R}}^p)$, 
 and we shall use the same letter 
 to denote the resulting representation.

For $b \in {\mathbb{Z}}$, 
 we define the following infinite-dimensional $K$-module:
\begin{equation}
\label{eqn:Xib}
\Xi(K,b):=
\bigoplus_{\substack{m,n \in {\mathbb{N}} \\[1pt] m-n \in 2{\mathbb{N}}+b}} 
 {\mathcal{H}}^m({\mathbb{R}}^p) \boxtimes {\mathcal{H}}^n({\mathbb{R}}^q)
\quad
\text{(algebraic direct sum).}  
\end{equation}
We recall from Howe--Tan \cite{xhowetan}:
\begin{proposition}
Suppose $\lambda \in A_+(p,q)$.  
Let $b$ and $\varepsilon$ be as in \eqref{eqn:b} and \eqref{eqn:e}.  
\begin{enumerate}
\item[{\rm{(1)}}]
There is a unique irreducible submodule of $I_{\varepsilon}(\lambda+\rho)$
 with $K$-types $\Xi(K,b)$.  
\item[{\rm{(2)}}]
There is a unique irreducible quotient of $I_{\varepsilon}(-\lambda+\rho)$
 with $K$-types $\Xi(K,b)$.  
\item[{\rm{(3)}}]
These two modules are isomorphic to each other.  
\end{enumerate}
\end{proposition}

\subsection{Discrete series for semisimple symmetric spaces}
\label{subsec:Xpq}

We equip ${\mathbb{R}}^{p+q}$ with the standard pseudo-Riemannian structure
\[
  g_{{\mathbb{R}}^{p,q}}
  :=
  dx_1^2 + \cdots + dx_p^2- dy_1^2-\cdots -dy_q^2.  
\]
Then $g_{{\mathbb{R}}^{p,q}}$ is nondegenerate
 on the following hypersurface
\[
   X(p,q) \equiv X(p,q)_+
   := \{(x,y) \in {\mathbb{R}}^{p+q}:
        |x|^2-|y|^2=1\}, 
\]
yielding a pseudo-Riemannian structure $g_{X(p,q)}$
 of signature $(p-1,q)$
 with constant sectional curvature $+1$, 
 sometimes referred to as a {\it{pseudo-Riemannian space form}}
 of positive curvature.  
We also set 
\[
   X(p,q)_-
   := \{(x,y) \in {\mathbb{R}}^{p+q}:
        |x|^2-|y|^2=-1\}.  
\]
Then $X(p,q)_-$ has a pseudo-Riemannian structure of signature $(p,q-1)$.  
There is a natural isomorphism
 (reversing the signature of the pseudo-Riemannian metric):
\[X(p,q)_- \simeq X(q,p)_+.  \]
Then $X(p,q)$ is a sphere $S^{p-1}$ if $q=0$,
 a hyperbolic space 
if $p=1$, 
 de Sitter manifold if $p=2$, 
and anti-de Sitter manifold if $q=1$.  
We note $X(0,q) = \emptyset$.

The group $G=O(p,q)$ acts isometrically
 and transitively on $X(p,q)_{\pm}$, 
 and we have $G$-diffeomorphims:
\[
  X(p,q)_+ \simeq O(p,q)/O(p-1,q),
  \quad
  X(p,q)_- \simeq O(p,q)/O(p,q-1).  
\]

The pseudo-Riemannian metric $g_{X(p,q)}$ induces the Radon measure, 
 and the Laplace--Beltrami operator $\Delta \equiv \Delta_{X(p,q)}$
 on $X(p,q)$.

For $\lambda \in {\mathbb{C}}$, 
 we consider a differential equation on $X(p,q)$:
\begin{equation}
\label{eqn:Laplmd}
\Delta_{X(p,q)} f =(-\lambda^2+\rho^2)f
\end{equation}
 where $\rho=\frac 1 2(p+q-2)$, 
 and set 
\begin{align*}
C^{\infty}(X(p,q))_{\lambda}
 :=&
 \{f \in C^{\infty}(X(p,q))
  :
  \text{$f$ satisfies \eqref{eqn:Laplmd} in the usual sense}\}, 
\\
L^2(X(p,q))_{\lambda}
 :=&
 \{f \in L^2(X(p,q))
  :
  \text{$f$ satisfies \eqref{eqn:Laplmd} in the distribution sense}\}.  
\end{align*}

\begin{proposition}
[Faraut \cite{xfar}, Strichartz \cite{xstri}]
\label{prop:discX}
$L^2(X(p,q))_{\lambda} \ne \{0\}$ 
 if and only if $\lambda \in A_+(p,q)$.  
\end{proposition}

The group $G=O(p,q)$ acts on $L^2(X(p,q))_{\lambda}$
 as an irreducible unitary representation.  
Moreover,
 if $f \in L^2(X(p,q))_{\lambda}$ is $K$-finite, 
 then there is an analytic function $a \in C^{\infty}(S^{p-1} \times S^{q-1})$
 such that 
\begin{equation}
\label{eqn:L2asym}
   f(\omega \cosh s, \eta \sinh s) 
   = 
   a (\omega, \eta) e^{-(\lambda+\rho)s}
\, 
   (1+s e^{-2s}O(1))
\quad
 \text{as $s \to \infty$.}
\end{equation}

\section{General scheme}
Our approach to the branching laws (Theorem \ref{thm:2002})
 is to use analysis on $G'$-orbits 
 in the reductive symmetric space $G/H$, 
 as developed in \cite{xkInvent94, xkdisc} among others.  
In our setting, 
 $G/H\simeq X(p,q)$ admits principal orbits of the subgroup $G'$ 
 (see \cite[Sect.~8.2]{xkdisc}), 
 hence all the discrete spectrum 
 in the branching law $\Pi|_{G'}$ can be captured
 though the analysis 
 on principal $G'$-orbits, 
 as formulated in Proposition \ref{prop:3.1.1} below.  

\subsection{Principal $G'$-orbits in $X(p,q)$}
\label{subsec:orbits}
We introduce a $G'$-invariant function in the ambient space
 $\Bbb R^{p+q} =\Bbb R^{p'+p'' + q'+q''}$ by
\begin{equation}
\label{eqn:level}
 \mu \colon \Bbb R^{p'+p'' + q'+q''} \to \Bbb R, 
    \ (u', u'', v', v'') \mapsto |u'|^2-|v'|^2.
\end{equation}
If $(u', u'', v', v'') \in X(p,q)$, 
 then
$$
    \mu(u', u'', v', v'') = |u'|^2-|v'|^2 = -|u''|^2 + |v''|^2 +1.
$$
We define three $G'$-invariant open sets 
 $X(p,q)_{\delta\varepsilon}$ of $X(p,q)$ by
\begin{align*}
 X(p,q)_\amp &:= 
         X(p,q) \cap \mu^{-1}(\set{s\in \Bbb R}{s<0}), 
\\
 X(p,q)_\app &:= 
         X(p,q) \cap \mu^{-1}(\set{s\in \Bbb R}{0<s<1}), 
\\
 X(p,q)_\apm &:=  
         X(p,q) \cap \mu^{-1}(\set{s\in \Bbb R}{1<s}).  
\end{align*}

Then the disjoint union 
\begin{equation}
\label{eqn:3union}
   X(p,q)_\amp \amalg X(p,q)_\app \amalg X(p,q)_\apm
\end{equation}
 is conull in $X(p,q)$.
Accordingly,
 we have a direct sum decomposition of the Hilbert space:
\begin{equation}
\label{eqn:2.2.1}
    L^2(X(p,q)) = L^2(X(p,q)_\amp)
                  \oplus L^2(X(p,q)_\app) \oplus L^2(X(p,q)_\apm),
\end{equation}
 which is stable by the action of $G'$.
We shall see in \eqref{eqn:mapmp}--\eqref{eqn:mappm}
 that the isomorphism classes
 of the isotropy subgroups 
 of the subgroup $G'$ at points 
 in $X(p,q)_{\delta\varepsilon}$ are determined uniquely
 by $(\delta, \varepsilon)$.

\subsection{A priori estimate of $\operatorname{Disc}(\Pi|_{G'})$}
By using the general theory \cite{xkdisc}, 
 we explain the three families
 of irreducible representations of $G'$ occurring 
 in the branching law $\Pi|_{G'}$
 (Theorem \ref{thm:2002})
 arise from the decomposition \eqref{eqn:3union}.
\begin{proposition}
\label{prop:3.1.1}
For $\lambda \in A_+(p,q)$, 
 we set $\Pi:=\pip{p}{q}{\lambda} \in \widehat G$
 as in Definition-Theorem \ref{def:pilmd}.  
If
 $\pi \in \widehat{G'}$ satisfies
 $\Hom_{G'}(\pi, \Pi|_{G'}) \neq \{0\}$,
 then there exist uniquely
 $(\delta', \delta'') \in \{\amp, \app, \apm \}$
 and 
 $(\lambda', \lambda'') \in A_{\delta'}(p', q')\times A_{\delta''}(p'', q'')$
 such that
\begin{equation}
\label{eqn:pikappa}
\pi \simeq
          \pia{\delta'}{p'}{q'}{\lambda'} \boxtimes
          \pia{\delta''}{p''}{q''}{\lambda''}.
\end{equation}
Moreover the following parity condition holds:
\begin{equation}
\label{eqn:parity}
   \delta' \lambda' + \delta'' \lambda''- \lambda
   \in 
   2 {\mathbb{Z}}+1.  
\end{equation}
\end{proposition}

\begin{proof}
The existence follows from the general results
 proved in \cite[Thm.8.6]{xkdisc}.  
The uniqueness is clear
 because these irreducible $G'$-modules are mutually inequivalent.  

To show the parity condition \eqref{eqn:parity}, 
 we observe
 that the central element $-I_{p,q}$ of $G$
 acts on $\pi_{\varepsilon,\lambda}^{p,q}$ as a scalar
$
  (-1)^{\lambda-\frac{p-q}{2}\varepsilon+1}, 
$
 as one sees from the equivalent condition (iii)
 in Definition-Theorem \ref{def:pilmd}.  
Since $(-I_{p',q'}) \times (-I_{p'',q''}) \in G'$ is identified
 with $-I_{p,q} \in G$, 
 it follows from the assumption $\Hom_{G'}(\pi, \Pi|_{G'})\ne \{0\}$ that
\[
    (-1)^{\lambda'-\frac{p'-q'}{2}\delta'+1}
    (-1)^{\lambda''-\frac{p''-q''}{2}\delta''+1}
    =
    (-1)^{\lambda-\frac{p-q}{2}+1}.  
\]
Then one obtains \eqref{eqn:parity}
 in view of 
$
  \lambda' \in {\mathbb{Z}}+\frac{p'+q'}{2},
\,\,
  \lambda'' \in {\mathbb{Z}}+\frac{p''+q''}{2}
\,\,
\text{ and }
\,\,
  \lambda \in {\mathbb{Z}}+\frac{p+q}{2}.  
$
\end{proof}

The above proof gives useful geometric information 
 on functions that belong to irreducible components
 of the branching law:
\begin{proposition}
\label{prop:3.1.2}
In the setting of Proposition \ref{prop:3.1.1},
 suppose $\pi \in \widehat{G'}$ satisfies 
 $\Hom_{G'}(\pi, L^2(X(p,q))_\lambda) \ne \{0\}$.  
We set $\kappa:=(\delta',\delta'') \in \{\amp, \app, \apm\}$
 according to \eqref{eqn:pikappa}
 in Proposition \ref{prop:3.1.1}.  
Then we have 
$$
  \operatorname{Supp} F \subset \overline{X(p,q)_{\kappa}}
$$
for any function $F$
 in the image of  $\Hom_{G'}(\pi, L^2(X(p,q))_\lambda)$.  
\end{proposition}

\section{Construction of holographic operators}

In this section we construct explicit intertwining operators
 ({\it{holographic operators}}) from 
 irreducible $G'$-modules 
 to irreducible $G$-modules:
\[
T_{\delta\varepsilon,\lambda}^{\lambda',\lambda''}
\colon
\pi_{\delta, \lambda'}^{p',q'} \boxtimes \pi_{\varepsilon, \lambda''}^{p'',q''} 
\to 
\pi_{+, \lambda}^{p,q}|_{G'},  
\]
by using a geometric realization 
 of these representations in the $L^2$-spaces
 of pseudo-Riemannian space forms
 $X(p',q')_{\delta}$, $X(p'',q'')_{\varepsilon}$ and $X(p,q)$,
 as described in Section \ref{subsec:Xpq}.  
Moreover,
 we find a closed formula for the operator norm
 of $T_{\delta\varepsilon,\lambda}^{\lambda',\lambda''}$.  
The main results of this section are stated in Theorem \ref{thm:holographic}.

\subsection{Preliminaries}
To state the quantitative results (Theorem \ref{thm:holographic}), 
 we set 
\begin{align*}
     \V{\apm}{\lambda'}{\lambda''}{\lambda}
     &:=
     \frac{
        (\Gamma(\lambda''+1))^2 
         \ 
         \Gamma(\frac{\lambda' - \lambda''+ \lambda+1}{2})
         \ 
         \Gamma(\frac{\lambda' - \lambda'' - \lambda+1}{2})
          }{
         2 \lambda
         \
         \Gamma(\frac{\lambda' + \lambda''+ \lambda+1}{2})
         \ 
         \Gamma(\frac{\lambda' + \lambda'' - \lambda+1}{2})
         }, 
\\
     \V{\app}{\lambda'}{\lambda''}{\lambda}
     &:=
     \frac{
        (\Gamma(\lambda''+1))^2 
         \ 
         \Gamma(\frac{-\lambda' - \lambda''+ \lambda+1}{2})
         \ 
         \Gamma(\frac{\lambda' - \lambda'' + \lambda+1}{2})
          }{
         2 \lambda
         \
         \Gamma(\frac{-\lambda' + \lambda''+ \lambda+1}{2})
         \ 
         \Gamma(\frac{\lambda' + \lambda'' + \lambda+1}{2})
         }, 
\\
      \V{\amp}{\lambda'}{\lambda''}{\lambda}
      &:=\V{\apm}{\lambda''}{\lambda'}{\lambda}.  
\end{align*}

\begin{lemma}
\label{lem:V}
\begin{enumerate}
\item[{\rm{(1)}}]
$\V {\delta\varepsilon}{\lambda'}{\lambda''}{\lambda}>0$
 if $\lambda>0$, 
 $\lambda', \lambda'' \ge -\frac 1 2$, 
 and $\delta \varepsilon \lambda -\varepsilon\lambda'-\delta\lambda'' >0$.  
Here $\delta \varepsilon \lambda:=\lambda$
 when $\delta=\varepsilon$
 and $-\lambda$ when $\delta \ne \varepsilon$.  
\item[{\rm{(2)}}]
$\V {\delta\varepsilon}{\lambda'}{\lambda''}{\lambda}>0$
 if $(\lambda', \lambda'') \in \Lambda_{\delta \varepsilon}(\lambda)$.  
\end{enumerate}
\end{lemma}
\begin{proof}
(1)\enspace
Clear from the definition.  
\itm{(2)}
The second statement is a special case of the first one.  
See also Lemma \ref{lem:2.4} for an alternative proof.  
\end{proof}

\subsection{Jacobi functions and Jacobi polynomials}

Let us consider the differential operator
\begin{equation}
\label{eqn:JacobiL}
L_{\apm}:=
\frac{d^2}{d t^2}
+((2 \lambda'+1) \tanh t+(2 \lambda''+1) \coth t)\frac{d}{d t}.  
\end{equation}
We recall 
 that for $\lambda$, $\lambda'$, $\lambda'' \in {\mathbb{C}}$
 with $\lambda'' \ne -1, -2, \cdots$, 
 the {\it{Jacobi function}} $\varphi_{i \lambda}^{(\lambda'', \lambda')}(t)$
 is the unique even solution 
 to the following differential equation
\begin{equation}
\label{eqn:JacobiODE}
 (L_{\apm}+((\lambda'+\lambda''+1)^2-\lambda^2)) \varphi=0
\end{equation}
such that $\varphi(0)=1$, 
 see Koornwinder \cite{xkoorn}, 
 for instance.  
We note that 
$
   \varphi_{i \lambda}^{(\lambda'', \lambda')}(t)
=
   \varphi_{-i \lambda}^{(\lambda'', \lambda')}(t).  
$
By the change of variables
 $z=-\sinh^2 t$, 
 $g(z):=\varphi(t)$ satisfies
 the hypergeometric differential equation
\begin{equation}
\label{eqn:4.4.6}
\left( 
z(1-z)\frac{\partial^2}{\partial z^2}
+ (c-(a+b+1)z) \frac{\partial}{\partial z}
 - a b \right) 
 g(z) = 0
\end{equation}
with
$$
  a= \frac{\lambda'+\lambda''+1-\lambda}2,\quad
  b= \frac{\lambda'+\lambda''+1+\lambda}2,\quad
  c=\lambda''+1.
$$
The hypergeometric differential equation \eqref{eqn:4.4.6} has
 a regular singularity $z=0$, 
 and its exponents are $0$, $-\lambda''$.
For $\lambda''\ne 0$, 
 we denote by $g_{1(0)}(z)$ and $g_{2(0)}(z)$
 the unique solutions 
 to \eqref{eqn:4.4.6}
 such that
\begin{equation}
\label{eqn:g0}
  g_{1(0)}(0)=1
  \quad
  \text{and}
  \quad
  \lim_{z \to 0}z^{\lambda''} g_{2(0)}(z)=1.  
\end{equation}
We set 
\[
   u_{j(0)}(t):=g_{j(0)}(-\sinh^2 t)
\qquad{\text{for $j=1,2$.}}
\]  
If $\lambda'' \ne 0, -1,-2, \cdots$, 
 then $u_{1(0)}(t)$ is the Jacobi function
 $\varphi_{i \lambda}^{\lambda'', \lambda'} (t)$
(see \eqref{eqn:JacobiF}), 
 and thus we have 
\begin{equation}
\label{eqn:JacobiF}
  \varphi_{i \lambda}^{(\lambda'', \lambda')}(t)
  =
  {}_2F_1\left(
       \tfrac{\lambda'+\lambda''+1-\lambda}2,
       \tfrac{\lambda'+\lambda''+1+\lambda}2;
       \lambda''+1;-\sinh^2 t
     \right), 
\end{equation}
where ${}_2F_1$ is the Gauss hypergeometric function.  
We need the following formul{\ae}
 for the $L^2$-norms
 of the Jacobi functions.  

\begin{lemma}
[{\cite[Lem.~8.2]{opq2}}]
\label{lem:2.4}
Suppose $\lambda >0$.  
\begin{alignat*}{2}
   &\int_0^\infty |\varphi_{i \lambda}^{(\lambda'', \lambda')}(t)|^2
       (\cosh t)^{2 \lambda'+1}
       (\sinh t)^{2 \lambda''+1}
   \ d t
   = V_{\apm,\lambda}^{(\lambda', \lambda'')}
\quad
 &&\text{if $(\lambda', \lambda'') \in \Lambda_{\apm}(\lambda)$.}  
\\
  &\int_0^\frac{\pi}2 |\varphi_{i \lambda}^{(\lambda'', \lambda')}(i \theta)|^2
       (\cos \theta)^{2 \lambda'+1}
       (\sin \theta)^{2 \lambda''+1}
   \ d \theta
   = V_{\app,\lambda}^{(\lambda', \lambda'')} 
\quad
&&\text{if $(\lambda', \lambda'') \in \Lambda_{\app}(\lambda)$.}
\end{alignat*}
\end{lemma}

\subsection{Construction of holographic operators}
\label{subsec:ho}

We define the following diffeomorphisms $\Phi_{\delta\varepsilon}$
 onto the open subsets $X(p,q)_{\delta\varepsilon}$ by
\begin{alignat}{2}
\label{eqn:mapmp}
\Phi_\amp \colon 
 &X(q', p') \times X(p'', q'') \times (0,\infty)
 &&\rarrowsim X(p,q)_\amp
\\
 &((y', x'), (x'', y''), t)
 &&\mapsto (x' \sinh t, x'' \cosh t, y' \sinh t, y'' \cosh t), 
\notag
\\
\label{eqn:mappp}
\Phi_\app \colon 
 &X(p',q') \times X(p'', q'') \times (0,\frac \pi 2)
 &&\rarrowsim X(p,q)_\app
\\
 &((x', y'), (x'', y''), \theta)
 &&\mapsto (x' \cos \theta, x'' \sin \theta, y' \cos \theta, y'' \sin \theta), 
\notag
\\
\label{eqn:mappm}
\Phi_\apm \colon 
 &X(p',q') \times X(q'', p'') \times (0,\infty)
 &&\rarrowsim X(p,q)_\apm
\\
 &((x', y'), (y'', x''), t)
 &&\mapsto (x' \cosh t, x'' \sinh t, y' \cosh t, y'' \sinh t).  
\notag
\end{alignat}

By using the following coordinates:
\begin{alignat*}{2}
   (z',z'', t) &= \Phi_{\delta \varepsilon}^{-1}(x) \ && \text{ for $x \in X(p,q)_{\delta \varepsilon}$ for $(\delta,\varepsilon)=(-,+)$ or $(+,-)$},
\\
 (z',z'', \theta) &= \Phi_\app^{-1}(x) \ && \text{ for $x \in X(p,q)_\apm$}, 
\end{alignat*}
we introduce linear operators
\begin{equation}
\label{eqn:T}
   T_{\delta\varepsilon, \lambda}^{\lambda', \lambda''} \colon
   L^2(X(p',q')_{\delta}) \widehat\otimes L^2(X(p'',q'')_{\varepsilon})
  \to L^2(X(p,q)), 
\end{equation}
as follows:
\begin{align*}
   T_{\amp, \lambda}^{\lambda', \lambda''} h (x)
   &:=
   \begin{cases}  
        h(z',z'') \ \varphi_{i \lambda}^{(\lambda', \lambda'')}(t)
          \ (\cosh t)^{\lambda''-\rho''}  (\sinh t)^{\lambda'-\rho'}
            & \text{if } x \in X(p,q)_\amp, 
\\
                  0 & \text{otherwise,}
   \end{cases}
\\
   T_{\app, \lambda}^{\lambda', \lambda''} h (x)
   &:=
   \begin{cases}
   h(z',z'') \ 
       \varphi_{i \lambda}^{(\lambda'', \lambda')}(i \theta)
          (\cos \theta)^{\lambda'-\rho'}  (\sin \theta)^{\lambda''-\rho''}
\,\,\,
            & \text{if } x \in X(p,q)_\app, 
\\
   0 & \text{otherwise,}
  \end{cases}
\\
T_{\apm, \lambda}^{\lambda', \lambda''} h (x)
   &:=
   \begin{cases}  h(z',z'') \ \varphi_{i \lambda}^{(\lambda'', \lambda')}(t)
          \ (\cosh t)^{\lambda'-\rho'}  (\sinh t)^{\lambda''-\rho''}
            & \text{if } x \in X(p,q)_\apm, 
\\
          0 & \text{otherwise.}
   \end{cases}
\end{align*}

\begin{theorem}
\label{thm:holographic}
Suppose $(\delta,\varepsilon)=(-,+)$, $(+,+)$ or $(+,-)$.  
Let $\lambda \in A_+(p,q)$
 and $(\lambda',\lambda'') \in \Lambda_{\delta\varepsilon}(\lambda)$.  
Then $T_{\delta\varepsilon, \lambda}^{\lambda', \lambda''}$ induces
 an injective $G'$-intertwining operator:
\[
  T_{\delta\varepsilon, \lambda}^{\lambda', \lambda''} \colon
   L^2(X(p',q')_{\delta})_{\lambda'} 
   \widehat\otimes 
   L^2(X(p'',q'')_{\varepsilon})_{\lambda''} 
   \to L^2(X(p,q))_{\lambda}.  
\]
Moreover, 
 $(\V {\delta\varepsilon}{\lambda}{\lambda'}{\lambda''})^{-\frac 1 2}
  T_{\delta\varepsilon, \lambda}^{\lambda', \lambda''}$ is an isometry.  
\end{theorem}

The proof of Theorem \ref{thm:holographic} is divided into two parts:
\begin{enumerate}
\item[$\bullet$]
to compute the operator norm of $T_{\delta \varepsilon,\lambda}^{\lambda', \lambda''}$, 
 see Proposition \ref{prop:opnorm};
\item[$\bullet$]
to show that $T_{\delta \varepsilon,\lambda}^{\lambda', \lambda''}h$ is 
 a weak solution
 to \eqref{eqn:Laplmd}, 
 see Proposition \ref{prop:weakSol}.  
\end{enumerate}

\subsection{Operator norms of the holographic operators}
We prove 
 that the linear operator $T_{\delta\varepsilon, \lambda}^{\lambda', \lambda''}$ is a scalar multiple
 of an isometric operator, 
 and find its $L^2$-norm.  
We do not need
 that $h$ satisfies a differential equation
 in the proposition below.  
\begin{proposition}
\label{prop:opnorm}
Suppose $(\delta, \varepsilon)=(-,+)$, $(+,+)$, or $(+,-)$.  
If $\lambda>0$ and $(\lambda', \lambda'') \in \Lambda_{\delta \varepsilon}(\lambda)$, 
 then $T_{\delta\varepsilon, \lambda}^{\lambda', \lambda''}$ is an isometry
 upto scaling:
\[
\| T_{\delta\varepsilon, \lambda}^{\lambda', \lambda''} h \|_{L^2(X(p,q))}^2
   =
   V_{\delta\varepsilon, \lambda}^{(\lambda', \lambda'')}
   \|h\|_{L^2(X(p',q')_{\delta} \times X(p'',q'')_{\varepsilon})}^2
\]
 for all $h \in L^2(X(p',q')_{\delta} \times X(p'',q'')_{\varepsilon})$.  
\end{proposition}

\begin{proof}
With respect to the diffeomorphisms \eqref{eqn:mapmp}--\eqref{eqn:mappm}, 
 the invariant measure $d \mu$
 on $X(p,q)$ is expressed as 
\begin{equation}
\label{eqn:measure}
  d \mu_{X(p,q)} = d \mu_{X(p',q')_{\delta}} d \mu_{X(p'',q'')_{\varepsilon}}
                   d \mu_{\delta\varepsilon}(t)
\quad
  \text{on $X(p,q)_{\delta\varepsilon}$}, 
\end{equation}
where
\begin{align}
   d \mu_{\amp}(t) :=& (\cosh t)^{2 \rho''+1}(\sinh t)^{2\rho'+1}d t,  
\notag
\\
   d \mu_{\app}(\theta) :=& (\cos \theta)^{2 \rho'+1}(\sin \theta)^{2\rho''+1}d \theta,  
\label{eqn:dmu++}
\\
\label{eqn:dmu+-}
d\mu_{+-}(t):=&
(\cosh t)^{2 \rho' +1} (\sinh t)^{2 \rho'' + 1} d t.  
\end{align}
Hence the proof of Proposition \ref{prop:opnorm} is reduced
 to Lemma \ref{lem:2.4}.  
\end{proof}

\subsection{Construction of smooth solutions on open sets}

Since the Laplacian $\Delta_{X(p,q)}$ is not
 an elliptic differential operator
 unless the signature of $g_{X(p,q)}$ is definite
 ({\it{i.e., }} $p=1$ or $q=0$), 
 eigenfunctions (in the distribution sense)
 of the Laplacian are not necessarily real analytic on $X(p,q)$.  
In fact, 
 when $p \ge 2$ and $q \ge 1$, 
 one sees from the proof of Corollary \ref{cor:discdeco}
 that $T_{\delta\varepsilon,\lambda}^{\lambda',\lambda''} h$ is never
 real analytic
 on the whole space $X(p,q)$
 if $h \not \equiv 0$ and $p' p'' \ne 0$.

We begin by considering the restriction
 of $T_{\delta\varepsilon,\lambda}^{\lambda',\lambda''} h$
 to the open set $X(p,q)_{\delta\varepsilon}$
 (Section \ref{subsec:orbits})
 for each $(\delta, \varepsilon)=(-,+)$, $(+,+)$, or $(+,-)$.  
\begin{proposition}
\label{prop:smoothSol}
Suppose  $\lambda, \lambda', \lambda'' \in {\mathbb{C}}$
 such that $\lambda', \lambda'' \ne -1,-2,\cdots$.  
Then for any $h \in C^{\infty}(X(p',q')_{\delta})_{\lambda'} \otimes C^{\infty}(X(p'',q'')_{\varepsilon})_{\lambda''}$, 
 $F(x):=T_{\delta\varepsilon, \lambda}^{\lambda',\lambda''} h(x)$
 satisfies the differential equation \eqref{eqn:Laplmd}
 on the open set $X(p,q)_{\delta \varepsilon}$.  
\end{proposition}

\begin{proof}
Suppose $(\delta, \varepsilon)=(+,-)$.  
We set
\begin{align*}
D_{\apm} =&
\frac{\partial^2}{\partial t^2}
   + ( (2 \rho'+1) \tanh t + (2\rho''+1) \coth t) \frac{\partial}{\partial t}, 
\\
L_{\apm} =&
 \frac{\partial^2}{\partial t^2}
    + ((2 \lambda'+1) \tanh t + (2 \lambda'' + 1) \coth t) \der{t},   
\end{align*}
where we set 
$
\rho' = \frac{p'+q'-2}2,
\
\rho'' = \frac{p''+q''-2}2.  
$
We note that 
$
\rho=\rho'+\rho''+1.  
$

A short computation shows that 
\[
  S_{\lambda', \lambda''}^{-1} \circ D_{\apm} \circ S_{\lambda', \lambda''}
  =
  L_{\apm} + ((\lambda' + \lambda''+1)^2-\rho^2
  - \frac{(\lambda')^2-(\rho')^2}{(\cosh t)^2} + \frac{(\lambda'')^2-(\rho'')^2}{(\sinh t)^2} ), 
\]
under the transform $S_{\lambda',\lambda''}$ defined by 
\begin{equation}
\label{eqn:S}
 (S_{\lambda', \lambda''} \varphi)(t) := (\cosh t)^{\lambda'-\rho'}(\sinh t)^{\lambda''-\rho''} \varphi(t).  
\end{equation}
Via the diffeomorphism
 $\Phi_{+-}$ \eqref{eqn:mappm}, 
 the Laplacian $\Delta_{X(p,q)}$ takes the form:
\begin{equation}
\label{eqn:4.4.2}
\Delta_{X(p,q)}
 =  - D_{\apm} + \frac{1}{\cosh^2 t} \Delta_{X(p', q')} 
  - \frac{1}{\sinh^2 t} \Delta_{X(q'', p'')}
\end{equation}
in $X(p,q)_{\apm}$.  
Therefore, 
 for nonzero $h' \in C^{\infty}(X(p',q'))_{\lambda'}$
 and $h'' \in C^{\infty}(X(q'',p''))_{\lambda''}$, 
 $F_{\apm}(z',z'',t):=h'(z') h''(z'') (S_{\lambda', \lambda''} \varphi)(t)$
 satisfies
\[
   (\Delta_{X(p,q)} + \lambda^2-\rho^2) F_{\apm} \circ \Phi_{\apm}^{-1}
   = 0 
   \quad
   \text{on $X(p,q)_{\apm}$}
\]
 if and only if $\varphi$ satisfies the Jacobi differential equation
 \eqref{eqn:JacobiODE}.  
Thus Proposition \ref{prop:smoothSol} is shown 
 for $(\delta,\varepsilon)=(+,-)$.

The proof for $(\delta,\varepsilon)=(-,+)$ is essentially the same,
 and that for $(\delta,\varepsilon)=(+,+)$ goes similarly.  
In this case, 
 the Laplacian takes the form:
\begin{equation*}
\Delta_{X(p,q)}
 = D_{\app}+ \frac{1}{\cos^2 \theta} \Delta_{X(p', q')} 
   + \frac{1}{\sin^2 \theta} \Delta_{X(p'', q'')}
\end{equation*}
 on $X(p,q)_{\app}$ in the coordinates via $\Phi_{\app}$, 
where we set 
\[
D_{++}:=\frac{\partial^2}{\partial \theta^2}
   - ((2 \rho'+1) \tan \theta  - (2\rho''+1) \cot \theta)
          \frac{\partial}{\partial \theta}.  
\]
By the change of variables $z=\sin^2 \theta$, 
 the function 
$$
    g(z', z'', z) 
   := (\cos \theta)^{-\lambda' + \rho'}
      (\sin \theta)^{-\lambda'' + \rho''}
     F \circ \Phi_\app(z', z'', \theta), 
$$
 satisfies 
 the same hypergeometric equation \eqref{eqn:4.4.6}, 
 with regular singularities:
 the exponents at $z=0$ are $0$, $-\lambda''$;
 and those at $z=1$ are $0$, $-\lambda'$.
\end{proof}

\subsection{Boundary $\partial X(p,q)_{\delta\varepsilon}$}
\label{subsec:bdry}

By definition \eqref{eqn:T}, 
 $T_{\delta \varepsilon, \lambda}^{\lambda', \lambda''}h$ is
 the extension of a solution
 to the differential equation \eqref{eqn:Laplmd}
 in the open domain $X(p,q)_{\delta \varepsilon}$
 (see Proposition \ref{prop:smoothSol})
 to the whole manifold $X(p,q)$ by zero
 outside the domain.  
In order to prove a precise condition for such an extension
 to give a weak solution to \eqref{eqn:Laplmd}
 in $L^2(X(p,q))$, 
 we need an estimate of the solution near the boundary.  

In this section
 we study the boundary $\partial X(p,q)_{\delta\varepsilon}$.  
We observe that 
\[
   \partial X(p,q)_{++}= \partial X(p,q)_{-+} \cup \partial X(p,q)_{+-}.  
\]
Since $\partial X(p,q)_{-+}$ is similar
 to $\partial X(p,q)_{+-}$, 
 we take a closer look at 
\begin{align*}
\partial X(p,q)_{+-}
=&\{(u',u'',v',v'')\in X(p,q): |u''|=|v''|\}, 
\intertext{which is a union of the following two submanifolds:}
\partial X(p,q)_{+-}^{\operatorname{sing}}
:=&\{(u',0,v',0): (u', v') \in X(p',q')\}, 
\\
\partial X(p,q)_{+-}^{\operatorname{reg}}
:=&\{(u',u'',v',v'')\in X(p,q) : |u''|=|v''| \ne 0 \}.  
\end{align*}

We note that the singular part 
 $\partial X(p,q)_{+-}^{\operatorname{sing}}$
 is diffeomorphic to $X(p',q')$
 and that the map $\Phi_{\apm}$ extended to $t=0$ 
 in \eqref{eqn:mappm}
 surjects $\partial X(p,q)_{+-}^{\operatorname{sing}}$:
\[
     \Phi_{\apm} (X(p',q') \times X(q'',p'') \times \{0\})
     =
     \partial X(p,q)_{\apm}^{\operatorname{sing}}.  
\]
On the other hand, 
 the regular part $\partial X(p,q)_{+-}^{\operatorname{reg}}$
 is a hypersurface in $X(p,q)$.  
In a neighbourhood $U$ of a point at $\partial X(p,q)_{\apm}^{\operatorname{reg}}$, 
 we set
\[
  \xi_1:=|v''|-|u''|, \,\, 
  \xi_2:=|v''|+|u''| \, (>0), 
\]
and take coordinates on $U$ $(\subset X(p,q))$ by
\begin{equation}
\label{eqn:coordxi}
   (u',u'',v',v'')
   =
  ((1+\xi_1 \xi_2)^{\frac 1 2}x',
    \frac 1 2 (\xi_2- \xi_1)\omega'', 
    (1+\xi_1 \xi_2)^{\frac 1 2}y', 
    \frac 1 2 (\xi_1 + \xi_2)\eta''), 
\end{equation}
where $z'=(x',y') \in X(p',q')$, 
 $\omega'' \in S^{p''-1}$, and $\eta' \in S^{q''-1}$.  
Then $U \cap X(p,q)_{\apm}$ is given by $\xi_1 >0$, 
 whereas $U \cap X(p,q)_{\app}$ is given by $\xi_1<0$.  

\begin{lemma}
\label{lem:1914194}
In the coordinates \eqref{eqn:coordxi}, 
 the Laplacian $\Delta_{X(p,q)}$ takes the form
\begin{equation}
\label{eqn:Lapxi}
\Delta_{X(p,q)}
=
\xi_1^2 \frac{\partial^2}{\partial \xi_1^2}
+
4 \frac{\partial^2}{\partial \xi_1 \partial \xi_2}
+\xi_1 P \frac{\partial}{\partial \xi_1}
+Q, 
\end{equation}
where $P$ and $Q$ are differential operators
 of variables $\xi_2$, $x'$, $y'$, $\omega''$ and $\eta''$
 with smooth coefficients.  
\end{lemma}
\begin{proof}
The coordinates \eqref{eqn:coordxi} are obtained from 
 $\Phi_{\apm}(z',z'',t)$, 
 see \eqref{eqn:mappm}, 
 successively by the following two steps:
\begin{align}
\label{eqn:coo1}
&\bullet \,\,\,
z''=(\omega''\sinh s, \eta''\cosh s) \in X(p'',q'')_-, 
\\
&\label{eqn:coo2}
\bullet \,\,\,
\xi_1=e^{-s}\sinh t, 
\,\,
\xi_2=e^s \sinh t.  
\end{align}
By change of coordinates in the first step, 
 the Laplacian $\Delta_{X(p,q)}$ takes the form \eqref{eqn:4.4.2}
 with the second term replaced by 
\[
  \frac 1 {\cosh^2 t}
  (-D^s +\frac 1 {\cosh^2 s} \Delta_{S^{q''-1}}-\frac 1 {\sinh^2 s}\Delta_{S^{p''-1}})
\]
where we set 
\[
  D^s := \frac{\partial^2}{\partial s^2}
         +
         ((q''-1) \tanh s+(p''-1)\coth s)
         \frac {\partial}{\partial s}.  
\]
Then the change of variables $(t,s) \mapsto (\xi_1, \xi_2)$
 in the second step yields
\begin{equation*}
  \frac{\partial}{\partial s} 
  =
  - \xi_1 \frac {\partial}{\partial \xi_1} + \xi_2 \frac {\partial}{\partial \xi_2},
\quad
   \frac{\partial}{\partial t} 
  =
  \left(\frac{1+\xi_1 \xi_2}{\xi_1 \xi_1}\right)^{\frac 1 2}
  \left(\xi_1\frac{\partial}{\partial \xi_1} 
  + \xi_2 \frac {\partial}{\partial \xi_2}\right), 
\end{equation*}
whence the lemma by short computations.  
\end{proof}

\subsection{Extension as a weak solution in $L^2(X(p,q))$}

The proof of Theorem \ref{thm:holographic} will be completed
 if the image of $T_{\delta\varepsilon, \lambda}^{\lambda',\lambda''}$
 gives weak solutions
 to the differential equation \eqref{eqn:Laplmd}.  
\begin{proposition}
\label{prop:weakSol}
Suppose $(\delta, \varepsilon)=(-,+)$, $(+,+)$, or $(+,-)$.  
Assume $(\lambda', \lambda'') \in \Lambda_{\delta\varepsilon}(\lambda)$.  
Then for any $h \in L^2(X(p',q')_{\delta})_{\lambda'} \widehat \otimes L^2(X(p'',q'')_{\varepsilon})_{\lambda''}$, 
 $F:=T_{\delta\varepsilon, \lambda}^{\lambda',\lambda''} h$ is a weak solution 
 to the differential equation \eqref{eqn:Laplmd} on $X(p,q)$.  
\end{proposition}

\begin{proof}
Since the Laplacian $\Delta$ is a closed operator on $L^2(X(p,q))$,
 and since $T_{\delta\varepsilon, \lambda}^{\lambda',\lambda''}$ is 
 a bounded operator by Proposition \ref{prop:opnorm},
 it suffices to prove the assertion 
 for a dense subspace of the Hilbert space.  
Thus we may and do assume
 that $h$ is a $K'$-finite function.  
Then $F$ is real analytic on $X(p,q)_{\delta \varepsilon}$
 and satisfies \eqref{eqn:Laplmd} in $X(p,q)_{\delta \varepsilon}$
 in the usual sense
 by Proposition \ref{prop:smoothSol}.

In order to prove 
 that $F$ is a weak solution to \eqref{eqn:Laplmd}
 in the whole manifold $X(p,q)$,  
 we consider the boundary $\partial X(p,q)_{\delta \varepsilon}$, 
 and explain the case $(\delta,\varepsilon)=(+,-)$.  
We may and do assume that $p''>0$.  
In fact, 
 if $p''=0$, 
 then $X(p,q)_{\app} = X(p,q)_{\apm} =\emptyset$
 and $T_{\apm, \lambda}^{\lambda',\lambda''}h |_{X(p,q)_{\apm}}$ extends
 to a smooth function on $X(p,q)$.

Suppose $p''>0$.  
Then $\lambda'' \in A_-(p'',q'')$ satisfies $\lambda''>0$.  
In order to prove that $F$ is a weak solution to \eqref{eqn:Laplmd}, 
 it suffices to verify it near the boundary
 $\partial X(p,q)_{\apm}=\partial X(p,q)_{\apm}^{\operatorname{reg}}
  \cup \partial X(p,q)_{\apm}^{\operatorname{sing}}$.  

\vskip 1pc
{\bf{Case I.}}\enspace
First, 
 we deal with a neighbourhood $U$
 of a point at $\partial X(p,q)_{\apm}^{\operatorname{reg}}$.  
We take coordinates of $U$ 
 as in \eqref{eqn:coordxi}.  
We recall 
 that the boundary $U \cap \partial X(p,q)_{\apm}$
 is given by $\xi_1=0$
 where $\xi_2 >0$.  
Then $\Phi_{\apm}(z',z'',t)$ 
 with $z''=(\omega'' \sinh s, \eta'' \cosh s)$,
 see \eqref{eqn:coo1}, 
approaches to boundary points
 in $\partial X(p,q)_{\apm}^{\operatorname{reg}}$, 
 as $t \to 0$ and $s \to \infty$
 with constraints 
\[
  C_1< e^{s} \sinh t< C_2
\qquad
\text{for some $0< C_1 < C_2$}, 
\]
because 
\[
  \xi_1 =e^{-s} \sinh t, \quad \xi_2 = e^s \sinh t.  
\]
Then it follows from \eqref{eqn:L2asym}
 that the $K'$-finite function $h$ has an asymptotic behavior
\begin{equation}
\label{eqn:hasym}
  h(z',z'')=a(z',\omega'',\eta'') e^{-(\lambda''+\rho'')s}(1+se^{-2s}O(1))
\end{equation}
as $s \to \infty$
 for some analytic function $a(z',\omega'',\eta'')$, 
 and therefore
 $F=T_{\apm,\lambda}^{\lambda',\lambda''}h$ in $U \cap X(p,q)_{\apm}$
 behaves as 
\[
  O(e^{-(\lambda''+\rho'')s}(\sinh t)^{\lambda''-\rho''})
 =O(\xi_1^{\lambda''}\xi_2^{-\rho''})
\]
near the boundary $\xi_1 \downarrow 0$, 
 whereas $F \equiv 0$ for $\xi_1 <0$.  
Since $\lambda''>0$ and since $\Delta_{X(p,q)}$ takes 
 the form \eqref{eqn:Lapxi}, 
 the distribution $\Delta_{X(p,q)}F$
 is actually a locally integrable function on $U$.  
Since $F$ solves \eqref{eqn:Laplmd}
 in $U \setminus \partial X(p,q)_{\apm}$ 
 in the usual sense, 
 so does $F$ in $U$ in the distribution sense.  

\vskip 1pc
{\bf{Case II.}}\enspace
Next, 
 we deal with a neighbourhood $U$
 of a point at $\partial X(p,q)_{\apm}^{\operatorname{sing}}$.  
In this case, 
 we use 
$
  (z',z'',t) \in X(p',q') \times X(q'', p'') \times [0, \infty)
$
 as coordinates of $U \cap \overline{X(p,q)_{\apm}}$
 via $\Phi_{\apm}$.

Since $F$ behaves as $O(t^{\lambda''-\rho''})$ 
 when $t$ tends to zero, 
 so does $Y_1 F$ as $O(t^{\lambda''-\rho''-1})$
 and $Y_1 Y_2 F$ as $O(t^{\lambda''-\rho''-2})$
 for any vector fields $Y_1$, $Y_2$ on $X(p,q)$.  
In view of the formula \eqref{eqn:dmu+-} of the measure $d \mu_{+-}(t)$, 
 these functions belong to $L_{\operatorname{loc}}^1({\mathbb{R}}, d \mu_{+-}(t))$
 if
\[
   (\lambda''-\rho''-2)+ (2\rho''+1)>-1, 
\]
which is automatically satisfied 
 because $\lambda'' >0$.  
Thus $F$ is a weak solution to \eqref{eqn:Laplmd}
 near the boundary $\partial X(p,q)_{\delta\varepsilon}$
 when $(\delta, \varepsilon)=(+,-)$.

The other cases $(\delta, \varepsilon)=(+,+)$ and $(-,+)$ are similar.  
Thus Proposition \ref{prop:weakSol} is proved.  
\end{proof}

\section{Exhaustion of holographic operators}
\label{sec:5}

Let $\Pi \in \widehat G$ be any discrete series representation
 for the pseudo-Riemannian space form
 $G/H\simeq X(p,q)$.  
In this section we prove
 that discrete spectra of the restriction 
 $\Pi|_{G'}$ are exhausted by \eqref{eqn:2.1.1}
 counted with multiplicities, 
 hence complete the proof of Theorem \ref{thm:2002}.

To be precise, 
 we recall from Proposition \ref{prop:discX}
 that any $\Pi \in \Disc{G/H}$ is of the form 
 $\Pi = \pi_{+,\lambda}^{p,q}$
 for some $\lambda \in A_+(p,q)$, 
 and from Proposition \ref{prop:3.1.1}
 that $\pi \in \widehat{G'}$ satisfying $\Hom_{G'}(\pi, \Pi|_{G'}) \ne \{0\}$
 must be of the form 
 $\pi=\pi_{\delta,\lambda'}^{p',q'} \boxtimes \pi_{\varepsilon,\lambda''}^{p'',q''}$
 for some $(\lambda',\lambda'')
 \in A_{\delta}(p',q') \times A_{\varepsilon}(p'',q'')$
 with $(\delta, \varepsilon) \in \{(-, +), (+, +), (+, -)\}$.  
We show
 that $(\lambda',\lambda'')$ is actually 
 an element of $\Lambda_{\delta\varepsilon}(\lambda)$.  
More strongly, 
 we prove:
\begin{theorem}
\label{thm:4.2}
Suppose that $\lambda \in A_+(p,q)$
 and 
 $(\lambda', \lambda'')
 \in A_{\delta}(p',q') \times A_{\varepsilon}(p'',q'')$.  
Then, we have
$$
\Hom_{G'}(\pi_{\delta, \lambda'}^{p',q'} \boxtimes \pi_{\varepsilon,\lambda''}^{p'',q''},
 \pip{p}{q}{\lambda}|_{G'})
  \simeq \begin{cases} \Bbb C T_{\delta\varepsilon, \lambda}^{\lambda', \lambda''} &
                \text{ if } (\lambda', \lambda'')  \in \Lambda_{\delta\varepsilon}(\lambda),
          \\
                0 & \text{ otherwise}.
         \end{cases}
$$
\end{theorem}

We already know in \cite{xk:1}
 that the direct sum \eqref{eqn:2.1.1} equals the whole restriction $\Pi|_{G'}$
 if $p'=0$ or $p''=0$.  
In this case, 
 $\Pi = \pi_{+, \lambda}^{p,q}$ is $K'$-{\it{admissible}}
 ({\it{cf}}. Section \ref{subsec:algdeco}), 
 and the multiplicity of each $K'$-type occurring in $\Pi$
 coincides with 
 that in \eqref{eqn:2.1.1}.  
Hence the restriction $\Pi|_{G'}$ is discretely decomposable
 and is isomorphic to the direct sum \eqref{eqn:2.1.1}.  
Thus, 
 we shall assume $p' p''>0$ from now on.

The rest of this section is devoted to the proof
 of Theorem \ref{thm:4.2}
 in the case $p' p'' >0$
 and $(\delta, \varepsilon)=(+,-)$.  
The other cases
 where $(\delta, \varepsilon)=(-,+)$ or $(+,+)$
 are similar.  

\subsection{Kummer's relation}
The hypergeometric differential equation \eqref{eqn:4.4.6}
 has a regular singularity 
 also at $z=\infty$, 
 and its exponents 
  are $\frac1 2 (\lambda'+\lambda''+1-\lambda)$
 and $\frac1 2 (\lambda'+\lambda''+1+\lambda)$.
Suppose $\lambda \ne 0$.  
We write $g_{(\infty)}^+(z)$ and $g_{(\infty)}^-(z)$
 for the unique solutions to \eqref{eqn:4.4.6}
 such that
\begin{equation}
\label{eqn:ginfty}
  \lim_{z \to \infty} (-z)^{\frac{\lambda'+\lambda''+1 \mp \lambda}{2}}
  g_{(\infty)}^{\pm}(z)=1, 
\end{equation}
and set
\begin{equation}
\label{eqn:uinfty}
 u_{(\infty)}^{\pm}(t):=g_{(\infty)}^{\pm}(-\sinh^2 t). 
\end{equation}

\begin{lemma}
[Kummer's relation]
\label{lem:Kummer}
Suppose $\lambda \ne 0, -1, -2, \dots$
 and $\lambda'' \ne 0$.  
\begin{enumerate}
\item[{\rm{(1)}}]
There exist uniquely
 $a(\lambda', \lambda'', \lambda)$, 
 $b(\lambda', \lambda'', \lambda) \in {\mathbb{C}}$
 such that
\begin{equation}
\label{eqn:gKummer}
  g_{(\infty)}^-(z)
  =
  a(\lambda', \lambda'', \lambda)g_{1(0)}(z)
  +
  b(\lambda', \lambda'', \lambda)e^{i \pi \lambda''}
  g_{2(0)}(z).  
\end{equation}
\item[{\rm{(2)}}]
If $\lambda'' \ne 0,-1,-2,\dots$, 
 then 
\begin{equation}
\label{eqn:1914110}
b(\lambda', \lambda'', \lambda)
=
  \frac{\Gamma(\lambda'') \Gamma(1+\lambda)}
       {\Gamma(\frac{-\lambda'+ \lambda''+ \lambda+1}2)
        \Gamma(\frac{\lambda'+ \lambda''+ \lambda+1}2)}.  
\end{equation}
Moreover, 
 if $\lambda'' \not\in{\mathbb{Z}}$, 
 then $a(\lambda', \lambda'', \lambda)=b(\lambda', -\lambda'', \lambda)$.  
\end{enumerate}
\end{lemma}

\begin{proof}
The first statement is clear 
because $g_{1(0)}(z)$ and $g_{2(0)}(z)$ are linearly independent
 solutions to \eqref{eqn:4.4.6}.

To see the second statement, 
 we begin with the generic case
 where $\lambda \not \in \{0,-1,-2,\cdots\}$
 and $\lambda'' \not \in {\mathbb{Z}}$.  
Then we have 
\begin{align}
g_{(\infty)}^-(z)
=&(-z)^{\frac{\lambda'+ \lambda''+\lambda+1}2}
  {}_2 F_1(\tfrac{\lambda'+ \lambda''+ \lambda+1}2, 
           \tfrac{\lambda'- \lambda''+ \lambda+1}2;
           1+\lambda;z^{-1}), 
\notag
\\
g_{1(0)}(z)
=&{}_2 F_1(\tfrac{\lambda'+ \lambda''- \lambda+1}2, 
           \tfrac{\lambda'+ \lambda''+ \lambda+1}2;
           1+\lambda'';z), 
\label{eqn:g10}
\\
g_{2(0)}(z)
=&z^{-\lambda''} {}_2 F_1(\tfrac{\lambda'- \lambda''- \lambda+1}2, 
           \tfrac{\lambda'- \lambda''+ \lambda+1}2;
           1-\lambda'';z), 
\label{eqn:g20}
\end{align}
 and Kummer's relation \cite[2.9 (39)]{xerd}
 shows 
$
  a(\lambda',\lambda'',\lambda)=b(\lambda',-\lambda'',\lambda) 
$
with the formula \eqref{eqn:1914110}
 for $b(\lambda',\lambda'',\lambda)$.

When $\lambda'' = m \in {\mathbb{N}}_+$, 
 $g_{1(0)}(z)$ remains to be the same \eqref{eqn:g10}
 but $g_{2(0)}(z)$ does not take the form \eqref{eqn:g20}.  
In fact, 
 $g_{2(0)}(z)$ contains a logarithmic term, 
 and is given by the analytic continuation:
\[
  \lim_{\lambda'' \to m}
  (g_{2(0)}(z)-\frac{P_m}{\lambda''-m} g_{1(0)}(z))
\]
where $P_m \equiv P_m(\lambda',\lambda) \in {\mathbb{C}}$
 is determined by
\[
  \lim_{\lambda'' \to m}
  (\lambda''-m)g_{2(0)}(z)
  =
  P_m g_{1(0)}(z).  
\]
Then the change of basis may alter
 the coefficient $a(\lambda',\lambda'',\lambda)$ 
 in \eqref{eqn:gKummer}
 but leaves $b(\lambda',\lambda'',\lambda)$ invariant.  
Thus the lemma is proved.  
\end{proof}

For $\lambda',\lambda'' \in {\mathbb{R}}$,
 we set a measure $d \mu^{\lambda',\lambda''}$ on ${\mathbb{R}}$
 by 
\[
  d \mu^{\lambda',\lambda''}(t)
   :=
  (\cosh t)^{2\lambda'+1}(\sinh t)^{2\lambda''+1} d t.  
\]
We note that $d \mu_{\apm}(t)=d \mu^{\rho',\rho''}(t)$, 
 see \eqref{eqn:dmu+-}, 
 and 
\begin{equation}
\label{eqn:SL2}
u \in L^2((0,\infty), d \mu^{\lambda',\lambda''}(t))
 \Leftrightarrow
 S_{\lambda',\lambda''}(u) \in L^2((0,\infty), d \mu_+(t))
\end{equation}
by the definition of the transform \eqref{eqn:S}
 of $S_{\lambda',\lambda''}$.

We need the following:
\begin{lemma}
\label{lem:KummerL2}
Suppose $\lambda>0$, $\lambda'>-1$, $\lambda''>-1$.  
Then 
$
  u_{(\infty)}^-(t) \in L^2((0, \infty), d \mu^{\lambda', \lambda''}(t))
$
 if and only if $-1 < \lambda''<1$
 or $\lambda'-\lambda''-\lambda-1 \in 2 {\mathbb{N}}$.  
\end{lemma}

\begin{proof}
By the asymptotic behavior \eqref{eqn:ginfty} of 
 $g_{(\infty)}^-(z)$ as $z \to \infty$, 
 we have
\[
  u_{(\infty)}^-(t)
 =g_{(\infty)}^-(-\sinh^2 t)
  \in 
  L^2([1, \infty), d \mu^{\lambda', \lambda''}(t))
\]
because $\lambda >0$.  
Likewise,
 by the asymptotic behavior \eqref{eqn:g0} of $g_{1(0)}(z)$
 and $g_{2(0)}(z)$
 as $z \to 0$, 
\begin{align*}
u_{1(0)}\in L^2((0, 1], d \mu^{\lambda', \lambda''}(t))
&\Leftrightarrow
\operatorname{Re}\lambda'' >-1,
\\
u_{2(0)}\in L^2((0, 1], d \mu^{\lambda', \lambda''}(t))
&\Leftrightarrow
\operatorname{Re}\lambda'' <1.  
\end{align*}
In view of the Kummer's relation \eqref{eqn:gKummer}, 
\[
   u_{(\infty)}^-(t)
   =a(\lambda', \lambda'', \lambda) u_{1(0)}(t)
   +
    b(\lambda', \lambda'', \lambda) u_{2(0)}(t)
\]
 belongs to $L^2((0, \infty), d \mu^{\lambda', \lambda''}(t))$
 if and only if 
 $-1< \lambda''<1$
 or $b(\lambda', \lambda'', \lambda)=0$.  
The latter condition amounts to 
 $\lambda'-\lambda''-\lambda-1 \in 2{\mathbb{N}}$
 by Lemma \ref{lem:Kummer} (2).  
Thus the lemma is proved.  
\end{proof}
\subsection{Possible form of holographic operators}
\label{subsec:4.4}
In this section 
 we examine a possible form for a holographic operator
 $\pi \to \Pi|_{G'}$, 
 and find a necessary condition 
 on the parameter for $\Hom_{G'}(\pi,\Pi|_{G'})$
 to be nonzero.  
We begin with the following:

\begin{lemma}
\label{lem:psi+-}
Let $\lambda \in A_+(p,q)$
 and $(\lambda',\lambda'')\in A_+(p',q') \times A_-(p'',q'')$.  
Suppose 
$
   T \in \Hom_{G'}(\pi_{+,\lambda'}^{p',q'} 
                   \boxtimes \pi_{-,\lambda''}^{p'',q''}, 
                   \pi_{+,\lambda}^{p,q}|_{G'})
$.  
Then in the geometric realizations
 of these representations
 on pseudo-Riemannian space forms
 (Section \ref{subsec:Xpq}), 
 $T$ must be of the following form:
 there exists $c \in {\mathbb{C}}$
 such that 
\[
T h =
\begin{cases}
c(h S_{\lambda',\lambda''}(u_{(\infty)}^-)) \circ \Phi_{\apm}^{-1}
\quad
&\text{on $X(p,q)_{\apm}$}, 
\\
0
&\text{otherwise}, 
\end{cases}
\]
for all $h \in L^2(X(p',q'))_{\lambda'} \widehat{\otimes} L^2(X(q'',p''))_{\lambda''}$. \end{lemma}
\begin{remark}
\label{rem:uinfty}
We have used the Jacobi function 
 $u_{1(0)}(t)=\varphi_{i \lambda}^{(\lambda'',\lambda')}(t)$
 \eqref{eqn:JacobiF}
 for the definition
 of the holographic operator
 $T_{\apm, \lambda}^{\lambda'',\lambda'}$ in \eqref{eqn:T}
 instead of $u_{(\infty)}^-(t)$
 as in Lemma \ref{lem:psi+-}.  
It is a part of Theorem \ref{thm:4.2}
 to show that $u_{1(0)}(t)$ is proportional 
 to $u_{(\infty)}^-(t)$
 if $(\lambda',\lambda'') \in \Lambda_+(\lambda)$.  
\end{remark}
\begin{proof}
[Proof of Lemma \ref{lem:psi+-}]
For any $h$ in $\pip{p'}{q'}{\lambda'} \boxtimes \pim{p''}{q''}{\lambda''}$, 
 we have $\operatorname{Supp} T h \subset \overline {X(p,q)_{+-}}$
 by Proposition \ref{prop:3.1.2}.

Suppose that $h$ is $K'$-finite.  
We set 
\begin{equation}
\label{eqn:psi+-}
\psi_{\apm} := S_{\lambda',\lambda''}^{-1} \circ T h \circ \Phi_{\apm}, 
\end{equation}
where $S_{\lambda',\lambda''}^{-1}$ 
 (see \eqref{eqn:S}) is applied to the last variable $t$.  
Then the following differential equations are satisfied:
\begin{multline*}
    \Delta_{X(p',q')} \psi_\apm 
    =(-(\lambda')^2 + (\rho')^2) \psi_\apm,\quad
  \Delta_{X(q',p')} \psi_\apm
  =(-(\lambda'')^2 + (\rho'')^2) \psi_\apm,
\end{multline*}
where 
$\Delta_{X(p',q')}$ acts on $z'$-variables,
 and 
$\Delta_{X(q',p')}$ on $z''$-variables.

As in the proof of Proposition \ref{prop:smoothSol}, 
 the differential equation \eqref{eqn:Laplmd} yields
 the following differential equation
 (in the sense of distribution):
\begin{equation}
\label{eqn:4.4.3}
 (L_\apm - (\lambda^2 - (\lambda' + \lambda'' + 1)^2))
  \psi_\apm(z', z'', t) = 0, 
\end{equation}
where $L_{\apm}$ is defined in \eqref{eqn:JacobiL}.  
Since $\lambda \ne 0$, 
 the solution $\psi_\apm(z', z'', t)$ is a linear combination
 of the basis $u_{(\infty)}^+(t)$ and $u_{(\infty)}^-(t)$.  
Hence $\psi_{\apm}$ is of the form
\[
\psi_{\apm}(z',z'',t)
=
h_+(z',z'') u_{(\infty)}^+(t)
+
h_-(z',z'') u_{(\infty)}^-(t)
\]
for some real analytic functions
 $h_+(z',z'')$ and $h_-(z',z'')$
 on $X(p',q') \times X(q'',p'')$. 
We observe that under the assumption $\lambda > 0$ 
 we have
\begin{equation}
 u_{(\infty)}^+(t) \not\in L^2([1, \infty); d \mu^{\lambda', \lambda''}(t)), 
\quad
 u_{(\infty)}^-(t) \in L^2([1, \infty); d \mu^{\lambda', \lambda''}(t) ).
\end{equation}
Since $\operatorname{Supp} T h \subset \overline{X(p,q)_{\apm}}$, 
 the formula \eqref{eqn:measure}
 of the invariant measure on $X(p,q)$
 and the definition \eqref{eqn:S}
 of $S_{\lambda',\lambda''}$ imply 
\[
  \|T h\|_{L^2(X(p,q))}^2
  =
  \int_{X(p',q') \times X(q'',p'')}
  \int_{0}^{\infty}
  |\psi_{\apm}(z',z'',t)|^2 
  d z' d z''
  d \mu^{\lambda', \lambda''}(t).  
\]
Thus
 we conclude from $T h \in L^2(X(p,q))$
 that $h_+(z',z'')=0$.  
In turn, 
 we have
\[
  \|T h\|_{L^2(X(p,q))}
  =
  \|h_-\|_{L^2 (X(p',q') \times X(q'',p''))}
  \|u_{(\infty)}^-\|_{L^2((0,\infty), d\mu^{\lambda',\lambda''}(t))}.  
\]
Since $T$ is a continuous map 
between the Hilbert spaces,
 we have
\begin{equation}
\label{eqn:u-}
  u_{(\infty)}^- (t) \in L^2((0,\infty), d\mu^{\lambda',\lambda''}(t))
\end{equation}
if $T \ne 0$.  
Moreover,
 $h \mapsto h_-$ is a $({\mathfrak{g}}',K')$-endomorphism
 of the irreducible $({\mathfrak{g}}',K')$-module
 $(\pi_{+,\lambda'}^{p',q'} \boxtimes \pi_{-,\lambda''}^{p'',q''})_{K'}$, 
 whence there exists $c \in {\mathbb{C}}$
 such that $h_- = c h$
 for all $K'$-finite vectors $h$
 by Schur's lemma.  
Since $T$ is a continuous map, 
 we obtain Lemma \ref{lem:psi+-}.  
\end{proof}

Next,
 we show 
 that the condition $T h \in L^2(X(p,q))$ leads us to the following:
\begin{proposition}
\label{prop:nec}
Retain $(\delta,\varepsilon)=(+,-)$.  
Suppose $\lambda \in A_+(p,q)$
 and $(\lambda',\lambda'')
 \in A_\delta(p',q') \times A_\varepsilon(p'',q'')$.  
If 
$
   \Hom_{G'}(\pi_{\delta,\lambda'}^{p',q'} \boxtimes \pi_{\varepsilon,\lambda''}^{p'',q''}, \pi_{+,\lambda}^{p,q}|_{G'}) \ne \{0\}$, 
then $\lambda'' =\frac 1 2$
 or $(\lambda', \lambda'') \in \Lambda_{\delta\varepsilon}(\lambda)$.  
\end{proposition}
In Section \ref{subsec:Heviside}, 
 we treat the case $\lambda''=\frac 1 2$.  

\begin{proof}
As we have seen \eqref{eqn:u-}
 in the proof of Lemma \ref{lem:psi+-}, 
 $u_{(\infty)}^-(t) \in L^2((0,\infty), d \mu^{\lambda',\lambda''}(t))$.  
Hence $-1 < \lambda'' < 1$
 or $\lambda'-\lambda''-\lambda-1 \in 2 {\mathbb{N}}$
 by Lemma \ref{lem:KummerL2}.  
Since $\lambda'' \in A_-(p'',q'')$ with $p''>0$ 
 (see \eqref{eqn:A-}), 
 the only possible $\lambda''$ with $\lambda''<1$
 is $\lambda''=\frac 1 2$.  
(We note that $\lambda''=-\frac 1 2$ occurs only
 when $(p'',q'')=(0,1)$.)  
Thus Proposition \ref{prop:nec} is proved.  
\end{proof}

\subsection{The case $\lambda''=\frac 1 2$}
\label{subsec:Heviside}

The case $\lambda''=\frac 1 2$ is delicate 
 because there exists a continuous $G'$-homomorphism
\[
  T \colon \pi_{+,\lambda'}^{p',q'} \boxtimes \pi_{-,\lambda''}^{p'',q''}
  \to 
  L^2(X(p,q)_{\apm})
\]
such that the image of $T$ consists of weak solutions
 to \eqref{eqn:Laplmd}
 in $L^2(X(p,q)_{\apm})$
 without the assumption 
 $(\lambda',\lambda'') \in \Lambda_{\apm}(\lambda)$.  
However, 
 we shall see 
 that $T h$ {\it{cannot}} be a weak solution
 to \eqref{eqn:Laplmd} in $L^2(X(p,q))$
 unless $(\lambda', \lambda'') \in \Lambda_{\apm}(\lambda)$.  
For this, 
 it suffices to show the following:
\begin{lemma}
\label{lem:Heviside}
In the setting of Lemma \ref{lem:psi+-}, 
 suppose $\lambda''=\frac 1 2$
 and $(\lambda', \lambda'') \not \in \Lambda_{\apm}(\lambda)$.  
Then the distribution $\Delta_{X(p,q)}(T h)$
 is not a locally integrable function on $X(p,q)$
 for any nonzero $K'$-finite function $h$.  
\end{lemma}

\begin{proof}
We consider a neighbourhood $U$ 
 at a point of $\partial X(p,q)_{\apm}^{\operatorname{reg}}$, 
 and use the coordinates \eqref{eqn:coordxi}
  as in Section \ref{subsec:bdry}.  
Then $T h=0$
 if $\xi_1<0$.  
Let us examine the behavior of $T h$
 in $U \cap \overline{X(p,q)_{\apm}}$
 near the boundary as $\xi_1 \downarrow 0$.

Let $\psi_{\apm}$ be as in \eqref{eqn:psi+-}.  
Since $(\lambda',\lambda'') \not \in \Lambda_{\apm}(\lambda)$, 
 the coefficient $b(\lambda',\lambda'',\lambda)$ 
 in \eqref{eqn:gKummer} does not vanish.  
Hence there exist $A \in {\mathbb{C}}$
 and $B \ne 0$
 such that 
\begin{align*}
   \psi_{\apm}(z',z'',t)
   =& h(z',z'') (A u_{1(0)}(t) + B u_{2(0)}(t))
\\
  =& h(z',z'') (A -B t^{-1})(1+O(t^2)).  
\end{align*}

We recall from \eqref{eqn:hasym}
 that $h(z',z'')$ has an asymptotic behavior 
\[
 h(z',z'')=a(z',\omega'',\eta'') e^{-(\lambda''+\rho'')s}(1+se^{-2s}O(1))
\]
 for some real analytic function
 of $(z',\omega'',\eta'') \in X(p',q') \times S^{p''-1} \times S^{q''-1}$
 as $s \to \infty$
 in the coordinates $z''=(\omega'' \sinh s, \eta'' \cosh s)$.

Combining these two asymptotic behaviours
 as $s \to \infty$ and $t \to 0$
 with $\xi_2 =e^s \sinh t$
 away from 0 and infinity, 
 we obtain the asymptotic behavior 
 of $T h$
 near the boundary $\partial X(p,q)_{\apm}^{\operatorname{reg}}$:
\[
  T h \sim \sum_{k=0}^{\infty}\xi_1^{\lambda''-\frac 1 2+\frac k 2} g_k(\xi_2, z', \omega'',\eta'')
\]
 where the first term is given by 
\[
  g_0=-B \xi_2^{-\frac 1 2 - \rho''}a(z',\omega'',\eta'').  
\]
In view of $\lambda''=\frac 1 2$, 
 the proof of the lemma is reduced to the following.  
\end{proof}

\begin{lemma}
\label{lem:191496}
Let $U$ be an open subset of ${\mathbb{R}}^n$, 
 and $P$ a differential operator on $U$ of the form
\[
  P=\xi_1^2 \frac{\partial^2}{\partial \xi_1^2} + \frac{\partial}{\partial \xi_1} P'+ P''
\]
such that $P'$ and $P''$ are differential operators
 of variables $\xi'=(\xi_2, \cdots, \xi_n)$
 with smooth coefficients
 in $\xi=(\xi_1, \xi')$.  
Suppose that $f(\xi)$ is a locally integrable function on $U$
 of the form
\[
  f(\xi)=
\begin{cases}
F(\xi_1^{\frac 1 2}, \xi')
\quad
&\text{for $\xi_1>0$,}
\\
0
\quad
&\text{for $\xi_1\le 0$, }
\end{cases}
\]
for some smooth function $F$.  
Then the distribution $P(\xi_1 f)$ is a continuous function in $U$.  
Furthermore, 
 $f$ is a weak solution to $P f=0$
 only when $P(\xi_1 f)|_{\xi_1=0} \equiv 0$.  
\end{lemma}

\begin{proof}
The first assertion is clear.  
Moreover we have 
$
   P(\xi_1 f)|_{\xi_1=0}= P_1 F(0,\xi').  
$

For the second assertion,
 we observe 
 that $f$ is a smooth function
 on $U^{\operatorname{reg}}:= U \setminus \{\xi_1 \ne 0\}$.  
Hence, 
in order to show $P f \ne 0$
 in the distribution sense, 
 it suffices to show
 that $P f$ does not belong to $L_{\operatorname{loc}}^1(U)$
 when $P_1 F(0,\xi') \not \equiv 0$.  
We introduce a locally integrable function $\widetilde f$ on $U$
 by 
\[
  \widetilde f(\xi):=
\begin{cases}
F(0, \xi')
\quad
&\text{for $\xi_1>0$,}
\\
0
\quad
&\text{for $\xi_1\le 0$.}
\end{cases}
\]
Clearly, 
 the distribution
\[
  \frac{\partial}{\partial \xi_1} P_1 \widetilde f
  =
  \delta(\xi_1) P_1 F(0,\xi')
\]
 is not locally integrable
 unless $P_1 F (0,\xi')\not \equiv 0$. 
Since $(P-\frac{\partial}{\partial \xi_1} P_1)f \in  L_{\operatorname{loc}}^1(U)$
 and $\frac{\partial}{\partial \xi_1} P_1 (f-\widetilde f)\in L_{\operatorname{loc}}^1(U)$, 
 we conclude 
 that $P f \not \in L_{\operatorname{loc}}^1(U)$.  
Thus the lemma is proved.  
\end{proof}

\section{Further analysis of the branching laws}
\label{sec:comments}

In this section we discuss further analytic aspects
 of the branching laws of the restriction $\Pi|_{G'}$
 of a discrete series representation $\Pi \in \operatorname{Disc}(G/H)$
 ($\subset \widehat G$), 
 see Section \ref{subsec:Pidisc} for notation.

\subsection{Generalities: discrete part of unitary representations}
\label{subsec:Pidisc}
Any unitary representation $\pi$
 of a reductive Lie group $L$
 has a unique irreducible decomposition:
\begin{equation}
\label{eqn:directint}
   \pi \simeq \int_{\widehat L} n_{\pi}(\sigma) \sigma \, d \mu (\sigma)
\qquad
\text{(direct integral)}, 
\end{equation}
where $d \mu$ is a Borel measure 
 on the unitary dual $\widehat L$, 
 and $n_{\pi} \colon \widehat L \to {\mathbb{N}} \cup \{\infty\}$
 is a measurable function
 ({\it{multiplicity}}).

In what follows, 
 we use the same letter 
 to denote a representation space 
 with the representation.  
Then the Hilbert direct sum 
\[
  \pi_{\operatorname{disc}} := \Hsum{\sigma \in \widehat L} \Hom_L(\sigma, \pi) \otimes \sigma
\]
is identified with the maximal closed $G$-submodule of $\pi$
 which is discretely decomposable.  
We say that the unitary representation $\pi_{\operatorname{disc}}$
 is the {\it{discrete part}}
 of the unitary representation $\pi$,  
and its orthogonal complement $\pi_{\operatorname{cont}}$ in $\pi$
 is the {\it{continuous part}} of $\pi$.

The unitary representation $\pi$ is discretely decomposable
 if $\pi = \pi_{\operatorname{disc}}$, 
 whereas $\pi=\pi_{\operatorname{cont}}$
({\it{i.e.}}, 
$\pi_{\operatorname{disc}}=\{0\}$)
 means
 that the irreducible decomposition \eqref{eqn:directint} does not 
 contain any discrete spectrum.

The irreducible decomposition \eqref{eqn:directint} is called
 the {\it{Plancherel formula}}
 when $\pi$ is the regular representation on $L^2(X)$
 where $X$ is an $L$-space
 with invariant measure;
 it is called the 
 {\it{branching law}}
 when $\pi$ is the restriction $\Pi|_L$
 of a unitary representation $\Pi$ of a group $G$
 containing $L$ 
 as a subgroup.  
The support
$
   \{\sigma \in \widehat L: \Hom_L(\sigma, \pi) \ne \{0\}\}
$
 will be denoted by
\begin{alignat*}{3}
&\operatorname{Disc}(X)\,\,
&&(\subset \widehat G)
\quad
&&\text{when $L=G$
 and $\pi$ is the regular representation $L^2(X)$;}
\\
&\operatorname{Disc}(\Pi|_{G'})\,\,
&&(\subset \widehat {G'})
\quad
&&\text{when $L=G'$ and $\pi$ is the restriction  $\Pi|_{G'}$.   }
\end{alignat*}

We consider the restriction $\Pi \in \operatorname{Disc}(G/H)$
 $(\subset \widehat G)$
 to the subgroup $G'$.  
The unitary representation $\Pi|_{G'}$ of the subgroup $G'$ splits
 into the discrete and continuous parts:
\[
  \Pi|_{G'} = (\Pi|_{G'})_{\operatorname{disc}} \oplus (\Pi|_{G'})_{\operatorname{cont}}.  
\]
We ask 
\begin{question}
\label{q:Disc}
Let $H$, $G'$ be reductive subgroups of $G$
 and $\Pi \in \Disc{G/H}$.  
\begin{enumerate}
\item[{\rm{(1)}}]
When $(\Pi|_{G'})_{\operatorname{disc}} =\{0\}$?
\item[{\rm{(2)}}]
When $\#(\operatorname{Disc}(\Pi|_{G'}))<\infty$?
\item[{\rm{(3)}}]
When $(\Pi|_{G'})_{\operatorname{cont}} =\{0\}$?
\end{enumerate}
\end{question}

We note that if $G'=H$
 and if $\Pi \in \operatorname{Disc}(G/H)$
 then the underlying $({\mathfrak{g}}, K)$-module $\Pi_K$
 is never discretely decomposable
 as a $({\mathfrak{g}}', K')$-module, 
 see \cite[Thm.~6.2]{xkInvent98}.  

\subsection{Criteria
 for $(\Pi|_{G'})_{\operatorname{disc}}=\{0\}$
 and $(\Pi|_{G'})_{\operatorname{cont}}=\{0\}$
}

We retain the previous setting 
where
\[\text{$G/H=O(p,q)/O(p-1,q) =X(p,q)$
 and 
$
  G'=O(p',q') \times O(p'',q'').  
$}
\]
{}From now, 
 we assume
\begin{equation}
\label{eqn:pqbasic}
   p=p'+p'' \ge 2, \, \, q=q'+q'' \ge 1, \,\, 
   \text{$(p',q') \ne (0,0)$ and $(p'',q'') \ne (0,0)$.}
\end{equation}
Then Proposition \ref{prop:discX} and Theorem \ref{thm:2002} 
 may be restated as:
\begin{align*}
\operatorname{Disc}(G/H)
=&
\{\pi_{+,\lambda}^{p,q}:
\lambda \in A_+(p,q)\}, 
\\
\operatorname{Disc}(\pi_{+,\lambda}^{p,q}|_{G'})
=&
\bigcup_{\delta, \varepsilon}
\{\pi_{\delta,\lambda'}^{p',q'} \boxtimes \pi_{\varepsilon,\lambda''}^{p'',q''}:
(\lambda', \lambda'') \in \Lambda_{\delta\varepsilon}(\lambda)\}. 
\end{align*}

In particular 
 $\operatorname{Disc}(G/H) \ne \emptyset$.

Here are answers to Question \ref{q:Disc} (1)--(3):
\begin{theorem}
[purely continuous spectrum]
\label{thm:conti}
The following two conditions on $(p',p'',q',q'')$ are equivalent:
\begin{enumerate}
\item[{\rm{(i)}}]
$\operatorname{Disc}(\Pi|_{G'}) = \emptyset$
 for any $\Pi \in \operatorname{Disc}(G/H)$;
\item[{\rm{(ii)}}]
$(p',p'')=(1,1)$,
 $(p',q')=(1,1)$ or $(p'',q'')=(1,1)$.  
\end{enumerate}
\end{theorem}

As a weaker property than Theorem \ref{thm:conti}, 
 we have:
\begin{theorem}
[at most finitely many discrete summands]
\label{thm:191419}
The following three conditions on $(p',p'',q',q'')$ are equivalent:
\begin{enumerate}
\item[{\rm{(i)}}]
$\# \operatorname{Disc}(\Pi|_{G'}) < \infty$
 for any $\Pi \in \operatorname{Disc}(G/H)$;
\item[{\rm{(ii)}}]
$\# \operatorname{Disc}(\Pi|_{G'}) < \infty$
 for some $\Pi \in \operatorname{Disc}(G/H)$;
\item[{\rm{(iii)}}]
$p' p''> 0$, 
 $\operatorname{min}(p'',q') \le 1$
 and $\operatorname{min}(p',q'') \le 1$.  
\end{enumerate}
\end{theorem}

As an opposite extremal case to Theorem \ref{thm:conti}, 
 we have:
\begin{theorem}
[discretely decomposable restriction]
\label{thm:discdeco}
The following three conditions on $(p',p'',q',q'')$ are  equivalent:
\begin{enumerate}
\item[{\rm{(i)}}]
The restriction $\Pi|_{G'}$ is discretely decomposable
 for any $\Pi \in \operatorname{Disc}(G/H)$;
\item[{\rm{(ii)}}]
The restriction $\Pi|_{G'}$ is discretely decomposable
 for some $\Pi \in \operatorname{Disc}(G/H)$;
\item[{\rm{(iii)}}]
$p'=0$ or $p''=0$.  
\end{enumerate}
\end{theorem}

For a unitary representation $\Pi$ of $G$, 
 the space $\Pi^{\infty}$ of smooth vectors
 (as a representation of $G$)
 is smaller in general than the space
 $(\Pi|_{G'})^{\infty}$ of smooth vectors
 as a representation of the subgroup $G'$.  
This difference detects discrete decomposability of the restriction $\Pi|_{G'}$  as follows.  
\begin{corollary}
\label{cor:discdeco}
Let $\Pi \in \Disc{G/H}$.  
Then the following two conditions are equivalent:
\begin{enumerate}
\item[{\rm{(i)}}]
The restriction $\Pi|_{G'}$ contains continuous spectrum
 in the branching law;
\item[{\rm{(ii)}}]
There does not exist a closed $G'$-irreducible submodule $W$ in $\Pi$
 such that $W \cap \Pi^{\infty} \ne \{0\}$.  
\end{enumerate}
\end{corollary}

\subsection{Proof of Theorem \ref{thm:191419}: finitely many summands}
We begin with the proof of Theorem \ref{thm:191419}.  

\begin{lemma}
\label{lem:191414}
In the setting \eqref{eqn:pqbasic}, 
 the following three conditions on $(p',p'',q',q'')$
 and $\lambda \in A_+(p,q)$ are equivalent:
\begin{enumerate}
\item[{\rm{(i)}}]
$\Lambda_{\apm} (\lambda) \ne \emptyset$;
\item[{\rm{(ii)}}]
$\# \Lambda_{\apm} (\lambda) = \infty$;
\item[{\rm{(iii)}}]
$p'' =0$ or \lq\lq{$p' \ge 2$ and $q'' \ge 2$}\rq\rq.  
\end{enumerate}
\end{lemma}

\begin{proof}
Direct from the definition of
 $\Lambda_{\apm} (\lambda)$ in Section \ref{sec:Intro}.  
\end{proof}

We note that the conditions (i) and (ii) in Lemma \ref{lem:191414}
 do not depend on the choice of $\lambda \in A_+(p,q)$.  
An analogous result holds for $\Lambda_{\amp}(\lambda)$
 by switching the role of $(p',q')$ and $(p'',q'')$.  
Hence we have:

\begin{lemma}
\label{lem:191417}
The following three conditions on $(p',p'', q',q'')$
 and $\lambda \in A_+(p,q)$ are equivalent:
\begin{enumerate}
\item[{\rm{(i)}}]
$\Lambda_{\amp} (\lambda) \cup \Lambda_{\apm} (\lambda) \ne \emptyset$; 
\item[{\rm{(ii)}}]
$\# (\Lambda_{\amp} (\lambda) \cup \Lambda_{\apm} (\lambda))= \infty$;
\item[{\rm{(iii)}}]
$p' p'' =0$, $\operatorname{min}(p'',q')\ge 2$, 
 or $\operatorname{min}(p',q'')\ge 2$.  
\end{enumerate}
\end{lemma}

Since $\# \Lambda_{\app} (\lambda)< \infty$ for any $\lambda$, 
 Theorem \ref{thm:191419} follows immediately from Lemma \ref{lem:191417}.  

\subsection{Nonexistence condition of discrete spectrum: proof of Theorem \ref{thm:conti}}

In this section,
 we discuss about when the restriction $\Pi|_{G'}$ decomposes into 
 continuous spectrum, 
 and give a proof of Theorem \ref{thm:conti}.

We begin with the following observation
 on elementary combinatorics:
\begin{lemma}
\label{lem:Aempty}
The condition (ii) in Theorem \ref{thm:conti} is equivalent to the condition:
\[
  A_{\delta}(p',q') \times A_{\varepsilon}(p'',q'') =\emptyset
\quad
\text{for $(\delta, \varepsilon)=(-,+)$, $(+,+)$ and $(+,-)$}.
\]
\end{lemma}
\begin{proof}
Clear from the definitions \eqref{eqn:A+} and \eqref{eqn:A-}
 of $A_{\pm}(p,q)$.  
\end{proof}

Thus the implication (ii) $\Rightarrow$ (i) in Theorem \ref{thm:conti}
follows readily from Theorem \ref{thm:2002} 
and Lemma \ref{lem:Aempty}.  

In order to prove the opposite implication,
 we need another elementary combinatorics as below.  
The proof is direct from the definition of $\Lambda_{\app}(\lambda)$.  
\begin{lemma}
\label{lem:191411}
In the setting \eqref{eqn:pqbasic}, 
 assume further that 
 $p', p'' \ge 2$.  
Then for $\lambda \in A_+(p,q)$, 
 we have the following:
\begin{enumerate}
\item[{\rm{(1)}}]
$\Lambda_{\app} (\lambda)=\emptyset$
\quad
if $\lambda < 2$ or if \lq\lq{$\lambda =2$ and $p' \equiv q' \mod 2$}\rq\rq;
\item[{\rm{(2)}}]
$\Lambda_{\app}(\lambda) \ne \emptyset$
\quad
if $\lambda>2$
 or if \lq\lq{$\lambda=2$ and $p' \not \equiv q' \mod 2$}\rq\rq.  
\end{enumerate}
\end{lemma}

We are ready to complete the proof
 of Theorem \ref{thm:conti}.  
\begin{proof}
[Proof of the implication (i) $\Rightarrow$ (ii) in Theorem \ref{thm:conti}]
Suppose that 
$
  \operatorname{Disc}(\Pi|_{G'}) =\emptyset
$
for any $\Pi \in \operatorname{Disc}(G/H)$.  
Then Theorem \ref{thm:191419} tells 
\begin{equation}
\label{eqn:thm72}
  p' p'' >0, \,\,
 \operatorname{min}(p'',q')\le 1, \,\,
 \text{and $\operatorname{min}(p',q'')\le 1$}.  
\end{equation}
On the other hand, 
 it follows from Lemma \ref{lem:191411} (2)
 that $\Lambda_{\app}(\lambda) \ne \emptyset$
 for $\lambda >2$ 
 if $\operatorname{min}(p',p'')\ge 2$.  
Hence we get $\min(p',p'') \le 1$.  
Without loss of generality, 
 we may and do assume $p'=1$.  
In turn, 
the condition \eqref{eqn:thm72} imply
\[
  (p', p'') = (1,1), \,
  (p',q') =(1,0), \,
  \text{ or }
  (p',q') =(1,1).  
\]
As we saw in Example \ref{ex:GP}, 
 $\operatorname{Disc}(\pi_{+,\lambda}^{p,q}|_{G'}) \ne \emptyset$
 for any $\lambda \in A_+(p,q)$ with $\lambda \ge 1$
 if $(p', q')=(1,0)$.  
Hence $(p', q')\ne(1,0)$.  
Thus the implication (i) $\Rightarrow$ (ii)
 in Theorem \ref{thm:conti} is proved.  
\end{proof}

\subsection{Proof of Theorem \ref{thm:discdeco}
 and Corollary \ref{cor:discdeco}}
\label{subsec:algdeco}

In the category of \gk-modules,
 analogous results to Theorem \ref{thm:discdeco}
 and Corollary \ref{cor:discdeco} are known in a general setting,  
 which we now recall:
\begin{proposition}
\label{prop:discdeco}
Let $(G,G')$ be a reductive symmetric pair.  
For $\Pi \in \widehat G$
 of which the underlying \gk-module $\Pi_K$ 
 is a Zuckerman derived functor module
 $A_{\mathfrak{q}}(\lambda)$.  
Then the following four conditions are equivalent:
\begin{enumerate}
\item[{\rm{(i)}}]
$\Pi_K$ is discretely decomposable as a $({\mathfrak{g}}',K')$-module
 (\cite[Def.~1.1]{xkInvent98}).  
\item[{\rm{(ii)}}]
$\Pi_K$ is $K'$-admissible,
 namely,
 $\dim_{\mathbb{C}} \Hom_{K'}(\tau, \Pi_K)<\infty$
 for any $\tau \in \widehat{K'}$.  
\item[{\rm{(iii)}}]
There exists a $G'$-irreducible closed subspace $\pi$ of $\Pi$
 such that $\pi \cap \Pi_K \ne \{0\}$.  
\item[{\rm{(iv)}}]
There exists a $G'$-irreducible closed subspace $\pi$ of $\Pi$
 such that $\pi \cap \Pi_K$ is dense in the Hilbert space $\pi$.  
\end{enumerate}
\end{proposition}
\begin{proof}
The equivalence (i) $\Leftrightarrow$ (ii) is proved
 in \cite[Thm.~4.2]{xkInvent98}. 
The equivalence (i) $\Leftrightarrow$ (iii) $\Leftrightarrow$ (iv)
 follows from \cite[Lem.~1.5]{xkInvent98}.  
\end{proof}

The equivalence holds 
 without the assumption 
 $\Pi_K \simeq A_{\mathfrak{q}}(\lambda)$.  
See also \cite{KOY15, K19}.

Back to our setting,
 we know from the classification theory \cite{decoAq}:

\begin{lemma}
\label{lem:K'adm}
The following three conditions on $(p', p'', q',q'')$ are equivalent:
\begin{enumerate}
\item[{\rm{(i)}}]
$\Pi_K$ is discretely decomposable as a $({\mathfrak{g}}',K')$-module
 for any $\Pi \in \Disc{G/H}$;
\item[{\rm{(ii)}}]
$\Pi_K$ is discretely decomposable as a $({\mathfrak{g}}',K')$-module
 for some $\Pi \in \Disc{G/H}$;
\item[{\rm{(iii)}}]
$p'=0$ or $p''=0$.  
\end{enumerate}
\end{lemma}

Since the discrete decomposability in the category of \gk-module
 implies the discrete decomposability 
 of the unitary representation, 
 the implication (iii) $\Rightarrow$ (i) ($\Rightarrow$ (ii))
in Theorem \ref{thm:discdeco} follows from Lemma \ref{lem:K'adm}.

To prove the converse implication (ii) $\Rightarrow$ (iii) in Theorem \ref{thm:discdeco}, 
 the following lemma is crucial.  
\begin{lemma}
\label{lem:Xiadm}
Let $G/H=O(p,q)/O(p-1,q)$ $(=X(p,q))$.  
Then the direct sum $\bigoplus_{\Pi \in \Disc{G/H}} \Pi$
 is $K$-admissible.  
\end{lemma}
\begin{proof}
This follows from the classification of $\Disc{G/H}$ 
 in Proposition \ref{prop:discX} 
 and from the $K$-type formula of $\Pi$ 
 as seen in the condition (iii) of Definition-Theorem \ref{def:pilmd}.  
\end{proof}

Combining Lemma \ref{lem:Xiadm} with Theorem \ref{thm:2002}, 
 we have
\begin{proposition}
\label{prop:Discadm}
For any $\Pi \in \Disc{G/H}$, 
 $(\Pi|_{G'})_{\operatorname{disc}}$ is $K'$-admissible.  
\end{proposition}

We are ready to complete the proof of Theorem \ref{thm:discdeco}.  
\begin{proof}
[Proof of the implication (ii) $\Rightarrow$ (iii) in Theorem \ref{thm:discdeco}]
Suppose that the restriction $\Pi|_{G'}$ is discretely decomposable
 as a unitary representation of the subgroup $G'$, 
 {\it{i.e.,}}
 $\Pi|_{G'}=(\Pi|_{G'})_{\operatorname{disc}}$.  
Then $\Pi$ is $K'$-admissible by Proposition \ref{prop:Discadm},
 and so is the underlying \gk-module $\Pi_K$.  
Hence $p'=0$ or $p''=0$ by Lemma \ref{lem:K'adm}.  
Thus Theorem \ref{thm:discdeco} is proved.  
\end{proof}

\begin{proof}
[Proof of Corollary \ref{cor:discdeco}]
By Theorem \ref{thm:discdeco}, 
 the condition (i) in Corollary \ref{cor:discdeco}
 is equivalent to the following:
\newline
(i)\enspace
$p' p'' \ne 0$, 
\newline
whereas the condition (ii) is clearly equivalent to 
\newline
(ii)$'$\enspace
For any $\pi \in \widehat {G'}$
 and any $\iota \in \Hom_{G'}(\pi, \Pi|_{G'})$, 
 $\iota(\pi) \cap \Pi^{\infty} = \{0\}$.  
\newline
Let us prove the equivalence (i)$'$ $\Leftrightarrow$ (ii)$'$.  
\newline
{(ii)$'$ $\Rightarrow$ (i)$'$:}\enspace
Suppose $p'p''=0$.  
Then $\iota(\pi) \cap \Pi_K \ne \{0\}$
 by Proposition \ref{prop:discdeco}, 
 whence $\iota(\pi) \cap \Pi^{\infty}\ne \{0\}$
 because $\Pi_K \subset \Pi^{\infty}$.  
\newline
{(i)$'$ $\Rightarrow$ (ii)$'$}:\enspace
Conversely,
 suppose $\iota\colon \pi \to \Pi|_{G'}$ is a nonzero continuous $G'$-homomorphism 
 for some $\pi \in \widehat{G'}$.  
Then $\pi$ must be of the form 
 $\pi_{\delta,\lambda'}^{p',q'}
 \boxtimes \pi_{\varepsilon,\lambda''}^{p'',q''}$
 for some $(\delta,\varepsilon)$ and $(\lambda',\lambda'')$, 
 and $\iota$ must be a scalar multiple
 of $T_{\delta\varepsilon, \lambda}^{\lambda',\lambda''}$
 by Theorems \ref{thm:2002} and \ref{thm:holographic}.  
If $p' p'' \ne 0$, 
 then it follows from the definition of $X(p,q)_{\delta\varepsilon}$ 
 in Section \ref{subsec:orbits} that at least two of the open sets
 $X(p,q)_{\amp}$, $X(p,q)_{\app}$, 
 $X(p,q)_{\apm}$ are nonempty,
 and thus 
$
   \operatorname{Image} 
   T_{\delta\varepsilon, \lambda}^{\lambda',\lambda''}
   \cap C^{\infty}(X(p,q)) =\{0\}
$
 by the definition of $T_{\delta\varepsilon, \lambda}^{\lambda',\lambda''}$
 in Section \ref{subsec:ho}.  
Since $\Pi^{\infty} \subset C^{\infty}(X(p,q))$, 
 this shows that $\iota(\pi) \cap \Pi^{\infty} =\{0\}$.  
Therefore, 
 we have shown the implication (i)$'$ $\Rightarrow$ (ii)$'$.  
\end{proof}

\section{Appendix ---multiplicity in branching laws}

As viewed in \cite{xKVogan2015}, 
 we divide branching problems into the following three stages:
\par\indent
\text{Stage A}:\enspace
Abstract features of the restriction; 
\par\indent
\text{Stage B}:\enspace
Branching laws 
 (irreducible decomposition of restrictions);
\par\indent
\text{Stage C}:\enspace
Construction of symmetry breaking/holographic operators.

The role of Stage A is to develop
 an abstract theory on the restriction
 of representations
 as generally as possible.  
In turn, 
 we could expect a detailed study of the restriction
 in Stages B and C
 in the specific settings
 that are {\it{a priori}} guaranteed
 to be \lq\lq{nice}\rq\rq\
 in Stage A.  
Conversely, 
 new results and methods in Stage C
 may indicate a further fruitful direction
 of branching problems
 including Stage A.

The present article has focused on analytic problems
 in Stages B and C
 in the setting where the triple $H \subset G \supset G'$ is given by 
\begin{equation}
\label{eqn:Opq3}
(G,H,G')=(O(p,q), O(p-1,q), O(p',q')\times O(p-p',q-q')).  
\end{equation}
Then one might wonder what are the abstract features (Stage A)
 which have arisen from this article, 
 and also might be curious about a possible generalization
 beyond the setting \eqref{eqn:Opq3}.  
The spectral property 
 of the branching laws is such an aspect, 
 which we discussed in Section \ref{sec:comments}.  
Another aspect of Theorem \ref{thm:2002} is 
 the multiplicity-free property:
\begin{equation}
\label{eqn:mult-one}
  m_{\Pi}(\pi) \le 1
\qquad
\text{
${}^{\forall} \pi \in \widehat {G'}$
 and
${}^{\forall} \Pi \in \operatorname{Disc}(G/H)$.  
}
\end{equation}
Here, 
 for $\Pi \in \widehat G$, 
 the multiplicity $m_{\Pi}(\pi)$ of $\pi \in \widehat {G'}$
 as the {\it{discrete spectrum}}
 of the (unitary) restriction $\Pi|_{G'}$
 is defined by
\[
m_{\Pi}(\pi):= \dim_{\mathbb{C}} \invHom {G'}{\pi}{\Pi|_{G'}}
=\dim_{\mathbb{C}}\invHom {G'}{\Pi|_{G'}}{\pi}
\in {\mathbb{N}} \cup \{\infty\}.  
\]

In this Appendix, 
 we give a flavor of some multiplicity estimates (Stage A)
 in a broader setting
 than \eqref{eqn:Opq3}, 
 for instance, 
 when
\begin{equation}
\label{eqn:triplesymm}
\text{both $(G,H)$ and $(G,G')$ are reductive symmetric pairs.}
\end{equation}
In what follows, 
we treat not only discrete series representations
 $\Pi \in \operatorname{Disc}(G/H)$
 but also non-unitary representations
 that have a non-trivial 
$H$-period (or is $H$-distinguished)
 as well.  
We recall that
 there is a canonical equivalence
 of categories 
 between the category ${\mathcal{H C}}$
 of $({\mathfrak{g}}, K)$-modules
 of finite length and the category ${\mathcal{M}}$
 of smooth admissible representations
 of moderate growth
  by the Casselman--Wallach globalization theory 
 \cite[Chap.~11]{WaI}.  
Denote by $\operatorname{Irr}(G)$
 the set of irreducible objects
 in ${\mathcal{M}}$.  
The unitary dual $\widehat G$ may be thought
 of as a subset of $\operatorname{Irr}(G)$
 by taking smooth vectors:
\begin{equation}
\label{eqn:Piinfty}
\widehat G \hookrightarrow \operatorname{Irr}(G), 
\qquad
\Pi \mapsto \Pi^{\infty}.  
\end{equation}

For $\Pi^{\infty} \in \operatorname{Irr}(G)$
 and $\pi^{\infty} \in \operatorname{Irr}(G')$, 
 we set
\[
m_{\Pi^{\infty}}(\pi^{\infty})
:= \dim_{\mathbb{C}} \invHom {G'}{\Pi^{\infty}|_{G'}}{\pi^{\infty}}.  
\]
In general, 
 for any $\Pi \in \widehat G$, 
one has 
 $m_{\Pi}(\pi) \le m_{\Pi^{\infty}}(\pi^{\infty})$
 for all $\pi \in \widehat {G'}$, 
 and 
 $m_{\Pi}(\pi) \le n_{\Pi}(\pi) \le m_{\Pi^{\infty}}(\pi^{\infty})$
 a.e.~$\pi \in \widehat {G'}$
 with respect to the measure
 for the disintegration \eqref{eqn:directint}
 of the (unitary) restriction $\Pi|_{G'}$, 
 where we recall
 $n_{\Pi} \colon \widehat {G'} \to {\mathbb{N}} \cup \{\infty\}$
 is the measurable function
 which gives the multiplicity
 in \eqref{eqn:directint}.

For a closed subgroup $H$ of $G$, 
 we define
\begin{equation*}
\operatorname{Irr}(G)_H
:=\{\Pi^{\infty} \in \operatorname{Irr}(G)
:(\Pi^{-\infty})^H \ne \{0\}
\}, 
\end{equation*}
where $\Pi^{-\infty}$ denotes the representation 
 on the space
 of distribution vectors.

Then $\operatorname{Disc}(G/H)$ may be thought of 
 as a subset of $\operatorname{Irr}(G)_H$
 via \eqref{eqn:Piinfty}.

Now we address the following:
\begin{prob}
\label{q:bdd}
Find a criterion for a triple $H \subset G \supset G'$
 with bounded multiplicity property
 for the restriction:
 there exists $C>0$
 such that
\begin{equation}
\label{eqn:BBH}
m_{\Pi^{\infty}}(\pi^{\infty}) \le C
\qquad
\text{${}^{\forall} \pi^{\infty} \in \operatorname{Irr}(G')$
 and ${}^{\forall} \Pi^{\infty} \in \operatorname{Irr}(G)_H$}.  
\end{equation}
\end{prob}

Note that the condition \eqref{eqn:BBH} immediately implies
\begin{equation}
\label{eqn:BBDisc}
m_{\Pi}(\pi) \le C  
\qquad
\text{
${}^{\forall}\pi \in \widehat {G'}$
 and 
${}^{\forall}\Pi \in \operatorname{Disc}(G/H)$.  
}
\end{equation}
We also note that \eqref{eqn:mult-one} is nothing but \eqref{eqn:BBDisc}
 with $C=1$.

We recall some general results in the setting 
 where $H=\{e\}$ from \cite[Thms.~C and D]{xktoshima}
 and \cite[Thm.~4.2]{xkInvent98}
 (see also Proposition \ref{prop:discdeco}):

{\bf{Bounded multiplicity:}}\enspace
$(G_{\mathbb{C}} \times G_{\mathbb{C}}')/\operatorname{diag} G_{\mathbb{C}}'$ is spherical 
 iff
\begin{equation}
\label{eqn:BB}
\text{${}^{\exists}C>0
\quad
m_{\Pi^{\infty}}(\pi^{\infty}) \le C
$
 \quad
 ${}^{\forall}\pi^{\infty} \in \operatorname{Irr}(G')$
 and ${}^{\forall}\Pi^{\infty} \in \operatorname{Irr}(G)$.  }
\end{equation}

{\bf{Finite multiplicity:}}\enspace
$(G \times G')/\operatorname{diag} G'$ is real spherical 
iff
\begin{equation}
\label{eqn:PP}
\text{
$m_{\Pi^{\infty}}(\pi^{\infty}) < \infty$
\quad
${}^{\forall} \pi^{\infty} \in \operatorname{Irr}(G')$
 and 
${}^{\forall}\Pi^{\infty} \in \operatorname{Irr}(G)$. 
}
\end{equation}

{\bf{Admissible restriction:}}\enspace
If $\Pi_K$ is discretely decomposable
 as a $({\mathfrak{g}}',K')$-module
 and if $(G,G')$ is a symmetric pair, 
 then 
\begin{equation}
\label{eqn:Wconj}
\text{$m_{\Pi}(\pi)=m_{\Pi^{\infty}}(\pi^{\infty})<\infty$
 for all $\pi \in \widehat{G'}$.  }
\end{equation}
(This generalizes Harish-Chandra's admissibility theorem
 for compact $G'$.)

In these cases, 
 explicit criteria lead us to the classification theory.  
The criterion \cite{xktoshima} for \eqref{eqn:BB} depends
 only on the complexification $({\mathfrak{g}}_{\mathbb{C}}, {\mathfrak{g}}_{\mathbb{C}}')$, 
 hence the classification for \eqref{eqn:BB}
 for simple ${\mathfrak{g}}_{\mathbb{C}}$ simple
 reduces to a classical result \cite{xkramer}:
\begin{equation}
\label{eqn:BBlist}
({\mathfrak{g}}_{\mathbb{C}}, {\mathfrak{g}}_{\mathbb{C}}')
=
({\mathfrak{sl}}_n, {\mathfrak{gl}}_{n-1}), 
({\mathfrak{so}}_{n}, {\mathfrak{so}}_{n-1}),
\text{ or } 
({\mathfrak{so}}_8, {\mathfrak{spin}}_7). 
\end{equation}
In this case, 
one can take $C=1$
 for most of the real forms
 \cite{xsunzhu}.  
On the other hand, 
 irreducible symmetric pairs 
 $({\mathfrak{g}}, {\mathfrak{g}}')$ satisfying \eqref{eqn:PP}
 were classified in \cite{xKMt}.  
The triples $(A_{\mathfrak{q}}(\lambda), {\mathfrak{g}}, {\mathfrak{g}}')$
 having discretely decomposable restrictions
 $A_{\mathfrak{q}}(\lambda)|_{{\mathfrak{g}}'}$ were classified
 in \cite{decoAq}.

We now consider the setting \eqref{eqn:triplesymm}.  
In this generality, 
 \eqref{eqn:BBDisc} may fail.  
The following example is a reformulation
 of \cite[Ex.~5.5]{xkAdv00}
 (cf. \cite[Sect.~6.3]{mf-korea}).  
\begin{example}
\label{ex:sp2C}
$(G,H,G')=(SO(5,{\mathbb{C}}),SO(3,2), SO(3,2))$.  
Then for any $\Pi \in \operatorname{Disc}(G/H)$
 there exists $\pi \in \widehat {G'}$
 such that $m_{\Pi}(\pi)=\infty$.  
(In this case, 
 the disintegration $\Pi|_{G'}$ contains
 continuous spectrum, 
 see \eqref{eqn:Wconj}.)
\end{example}

As we shall see in Observation \ref{obs:0.6} (1) below, 
 the bounded multiplicity property \eqref{eqn:BBH} often holds
 if $\operatorname{rank}G/H=1$, 
 but not always: 

\begin{example}
\label{ex:SU3}
Let $(G,H,G')=(SU(3),U(2), SO(3))$.  
Then \eqref{eqn:BBDisc} fails 
 because $m_{\Pi_n}(\pi_n)=[\frac n 2]+1$
 where $\Pi_n \in \operatorname{Disc}(G/H)$
 and $\pi_n \in \widehat{G'}$
 are of dimensions $(n+1)^3$ and $2n+1$, 
 respectively.  
\end{example}

\begin{example}
\label{ex:SL3}
Let $(G,H,G')=(SL(3,{\mathbb{R}}),GL(2,{\mathbb{R}}), SO(3))$.  
Then \eqref{eqn:BBDisc}  fails
 because $\sup_{\pi \in \widehat {G'}} m_{\Pi}(\pi)=\infty$
 for any $\Pi \in \operatorname{Disc}(G/H)$.  
\end{example}

The last example may be compared with the following:
\begin{example}
Let $(G,H,G')=(S L(4,{\mathbb{R}}), S p(2,{\mathbb{R}}), S O(4))$.  
Then \eqref{eqn:BBDisc} holds 
because $\sup_{\pi \in \widehat{G'}} m_{\Pi}(\pi) =1$
 for any $\Pi \in \operatorname{Disc}(G/H)$.  
\end{example}

To describe an answer to Problem \ref{q:bdd} (Stage A)
 which covers not only discrete series representations
 $\Pi \in \operatorname{Disc}(G/H)$
 but also any irreducible representations $\Pi^{\infty}$ 
realized in $C^{\infty}(G/H)$, 
 we fix some notation.  
Denote by $\sigma$ the involution of $G$
 that defines a symmetric pair $(G,H)$.  
We use the same letter $\sigma$
 to denote the complex linear extension of its differential.  
We write $G_{\mathbb{C}}$
 for a complexification of $G$, 
 and $G_U$ for a compact real form of $G_{\mathbb{C}}$.  
Let ${\mathfrak{j}}_{\mathbb{C}}$ be a maximal semisimple abelian
 subspace 
 in ${\mathfrak{g}}_{\mathbb{C}}^{-\sigma}=\{X \in {\mathfrak{g}}_{\mathbb{C}}:\sigma X=-X\}$, 
 and $Q_{\mathbb{C}}$ a parabolic subgroup 
 of $G_{\mathbb{C}}$
 with Levi part $Z_{G_{\mathbb{C}}}({\mathfrak{j}}_{\mathbb{C}})$.  

\begin{theorem}
\label{thm:bdd}
Suppose that $(G,H)$ is a reductive symmetric pair, 
 and $G'$ an (algebraic) reductive subgroup of $G$.  
Then the following three conditions on the triple $(G,H,G')$
 are equivalent:
\begin{enumerate}
\item[{\rm{(i)}}]
${}^{\exists}C>0$, 
 $m_{\Pi^{\infty}}(\pi^{\infty}) \le C$
\quad
 ${}^{\forall}\Pi^{\infty} \in \operatorname{Irr}(G)_H$
 and ${}^{\forall}\pi^{\infty} \in \operatorname{Irr}(G')$.  
\item[{\rm{(ii)}}]
$G_{\mathbb{C}}/Q_{\mathbb{C}}$ is $G_{\mathbb{C}}'$-spherical.  
\item[{\rm{(iii)}}]
$G_{\mathbb{C}}/Q_{\mathbb{C}}$ is $G_U'$-strongly visible.  
\end{enumerate}
\end{theorem}

See \cite{tanaka} 
 (see also \cite[Cor.~15]{xrims40})
for the equivalence (ii) $\Leftrightarrow$ (iii).

\begin{remark}
 The multiplicity-freeness 
 \eqref{eqn:mult-one} holds for compact forms.  
\end{remark}

It should be mentioned that the bounded multiplicity property (i)
 depends {\it{a priori}} 
 on the real form $(G,H,G')$, 
 however, 
Theorem \ref{thm:bdd} tells that its criterion (ii)
 (or equivalently (iii))
 can be stated only by the complexification
 of the Lie algebras $({\mathfrak{g}}, {\mathfrak{h}}, {\mathfrak{g}}')$.  
Here is a complete classification
 of such triples $({\mathfrak{g}}_{\mathbb{C}}, {\mathfrak{h}}_{\mathbb{C}}, {\mathfrak{g}}_{\mathbb{C}}')$
 when ${\mathfrak{g}}_{\mathbb{C}}$ is simple:

\begin{corollary}
[classification]
\label{cor:bdd}
Assume ${\mathfrak{g}}_{\mathbb{C}}$ is simple in the setting \eqref{eqn:triplesymm}.  
Then the bounded multiplicity property \eqref{eqn:BBH}
 holds for the triple $(G,H,G')$ 
 iff the complexified Lie algebras 
 $({\mathfrak{g}}_{\mathbb{C}}, {\mathfrak{h}}_{\mathbb{C}}, {\mathfrak{g}}_{\mathbb{C}}')$ 
 are in Table \ref{tab:0.1}
 up to automorphisms.  
In the table, 
 $p$, $q$ are arbitrary subject to $n=p+q$.

\begin{table}[H]
\begin{minipage}[t]{.45\textwidth}
\begin{tabular}[t]{ccc}
${\mathfrak{g}}_{\mathbb{C}}$
&${\mathfrak{h}}_{\mathbb{C}}$
&${\mathfrak{g}}_{\mathbb{C}}'$
\\
\hline
${\mathfrak{sl}}_n$
&${\mathfrak{gl}}_{n-1}$
&${\mathfrak{sl}}_p \oplus {\mathfrak{sl}}_q \oplus {\mathbb{C}}$
\\
${\mathfrak{sl}}_{2m}$
&${\mathfrak{gl}}_{2m-1}$
&${\mathfrak{sp}}_m$
\\
${\mathfrak{sl}}_{6}$
&${\mathfrak{sp}}_{3}$
&${\mathfrak{sl}}_4 \oplus {\mathfrak{sl}}_2 \oplus {\mathbb{C}}$
\\
${\mathfrak{so}}_n$
&${\mathfrak{so}}_{n-1}$
&${\mathfrak{so}}_p \oplus {\mathfrak{so}}_q$
\\
${\mathfrak{so}}_{2m}$
&${\mathfrak{so}}_{2m-1}$
&${\mathfrak{gl}}_m$
\\
${\mathfrak{so}}_{2m}$
&${\mathfrak{so}}_{2m-2} \oplus {\mathbb{C}}$
&${\mathfrak{gl}}_m$
\\
${\mathfrak{sp}}_n$
&${\mathfrak{sp}}_{n-1} \oplus {\mathfrak{sp}}_1$
&${\mathfrak{sp}}_p \oplus {\mathfrak{sp}}_q$
\\
${\mathfrak{sp}}_n$
&${\mathfrak{sp}}_{n-2} \oplus {\mathfrak{sp}}_2$
&${\mathfrak{sp}}_{n-1} \oplus {\mathfrak{sp}}_1$
\\
${\mathfrak{e}}_6$
&${\mathfrak{f}}_4$
&${\mathfrak{so}}_{10} \oplus {\mathbb{C}}$
\\
${\mathfrak{f}}_4$
&${\mathfrak{so}}_{9}$
&${\mathfrak{so}}_9$
\\
\end{tabular}
  \end{minipage}
  \hfill
  \begin{minipage}[t]{.45\textwidth}
\begin{tabular}[t]{ccc}
${\mathfrak{g}}_{\mathbb{C}}$
&${\mathfrak{h}}_{\mathbb{C}}$
&${\mathfrak{g}}_{\mathbb{C}}'$
\\
\hline
${\mathfrak{sl}}_n$
&${\mathfrak{so}}_{n}$
&${\mathfrak{gl}}_{n-1}$
\\
${\mathfrak{sl}}_{2m}$
&${\mathfrak{sp}}_{m}$
&${\mathfrak{gl}}_{2m-1}$
\\
${\mathfrak{sl}}_n$
&${\mathfrak{sl}}_p \oplus {\mathfrak{sl}}_q \oplus {\mathbb{C}}$
&${\mathfrak{gl}}_{n-1}$
\\
${\mathfrak{so}}_{n}$
&${\mathfrak{so}}_{p} \oplus {\mathfrak{so}}_q$
&${\mathfrak{so}}_{n-1}$
\\
${\mathfrak{so}}_{2m}$
&${\mathfrak{gl}}_{m}$
&${\mathfrak{so}}_{2m-1}$
\\
\end{tabular}
  \end{minipage}
\caption{Triples $({\mathfrak{g}}_{\mathbb{C}}, {\mathfrak{h}}_{\mathbb{C}}, {\mathfrak{g}}_{\mathbb{C}}')$ with ${\mathfrak{g}}_{\mathbb{C}}$ simple in Theorem \ref{thm:bdd}}
\label{tab:0.1}
\hfil
\end{table}

\end{corollary}

Here by \lq\lq{automorphisms}\rq\rq\
 we mean inner automorphisms
 for $({\mathfrak{g}}, {\mathfrak{h}})$
 and $({\mathfrak{g}}, {\mathfrak{g}}')$
 separately
 and outer autormorphisms for $({\mathfrak{g}}, {\mathfrak{h}}, {\mathfrak{g}}')$
 simultaneously.  
Thus in Table \ref{tab:0.1}, 
 we have omitted some cases such as 
 $({\mathfrak{g}}_{\mathbb{C}}, {\mathfrak{g}}_{\mathbb{C}}')
=
({\mathfrak{s o}}_8, {\mathfrak{spin}}_7)$, 
 $({\mathfrak{g}}_{\mathbb{C}}, {\mathfrak{h}}_{\mathbb{C}}, {\mathfrak{g}}_{\mathbb{C}}')=({\mathfrak{s o}}_8, {\mathfrak{g l}}_4, {\mathfrak{s o}}_6 \oplus {\mathfrak{s o}}_2)$
 or 
$({\mathfrak{s l}}_4, {\mathfrak{s p}}_2, {\mathfrak{s l}}_2 \oplus {\mathfrak{s l}}_2 \oplus {\mathbb{C}})$.

The right-hand side of Table \ref{tab:0.1} collects
 the case \eqref{eqn:BBlist}, 
 where a stronger bounded multiplicity theorem \eqref{eqn:BB} holds.  
The left-hand side includes:

\begin{example}
\label{ex:opqA}
The setting \eqref{eqn:Opq3} for Theorem \ref{thm:2002} is
 a real form of 
$
    ({\mathfrak{g}}_{\mathbb{C}}, {\mathfrak{h}}_{\mathbb{C}}, {\mathfrak{g}}_{\mathbb{C}}')
=({\mathfrak{s o}}_n, {\mathfrak{s o}}_{n-1}, {\mathfrak{s o}}_p \oplus {\mathfrak{s o}}_q)
$
 in the fourth row of the left-hand side in Table \ref{tab:0.1}.  
\end{example}

{}From Corollary \ref{cor:bdd}, 
 one sees the following:
\begin{observation}
\label{obs:0.6}
{\rm{(1)}}\enspace
The bounded multiplicity \eqref{eqn:BBH} holds
 for any triple 
 $(G,H,G')$ with $\operatorname{rank}G/H=1$
 except for the following two cases:
$
   ({\mathfrak{g}}_{\mathbb{C}}, {\mathfrak{h}}_{\mathbb{C}}, {\mathfrak{g}}_{\mathbb{C}}')
=({\mathfrak{s l}}_n, {\mathfrak{g l}}_{n-1}, {\mathfrak{s o}}_n)
\text{ or }
({\mathfrak{f}}_4, {\mathfrak{s o}}_9, 
 {\mathfrak{s p}}_3 \oplus {\mathfrak{s l}_2).  }
$
\newline
{\rm{(2)}}\enspace
The bounded multiplicity \eqref{eqn:BBH} may hold even
 when $\operatorname{rank}G/H >1$ and $\operatorname{rank}G/G' >1$.  
\end{observation}

Theorem \ref{thm:bdd} also gives a criterion for two reductive symmetric pairs $(G,H_1)$ and $(G, H_2)$
 with the following bounded multiplicity property 
 of tensor product representations.  
\begin{theorem}
[tensor product]
\label{thm:tensor}
Suppose that $(G,H_j)$ $(j=1,2)$ are reductive symmetric pairs, 
 and that ${Q_{j}}_{\mathbb{C}}$ are parabolic subgroups
 of $G_{\mathbb{C}}$ 
 as in Theorem \ref{thm:bdd}.  
Then the following three conditions on the triple 
 $(G,H_1, H_2)$ are equivalent:
\begin{enumerate}
\item[{\rm{(i)}}]
There exists $C>0$
such that 
\begin{equation}
\label{eqn:bddt}
  \dim_{\mathbb{C}} \operatorname{Hom}_G(\Pi_1 \otimes \Pi_2, \Pi) \le C
\quad
{}^{\forall}\Pi_j \in \operatorname{Irr}(G)_{H_j}\,(j=1,2)
\text{ and }
{}^{\forall}\Pi \in \operatorname{Irr}(G).  
\end{equation}
\item[{\rm{(ii)}}]
$(G_{\mathbb{C}} \times G_{\mathbb{C}})/({Q_{1}}_{\mathbb{C}} \times {Q_{2}}_{\mathbb{C}})$ is $G_{\mathbb{C}}$-spherical via the diagonal action.  
\item[{\rm{(iii)}}]
$(G_{\mathbb{C}} \times G_{\mathbb{C}})/({Q_{1}}_{\mathbb{C}} \times {Q_{2}}_{\mathbb{C}})$ is $G_U$-strongly visible via the diagonal action.  
\end{enumerate}
\end{theorem}

By the classification of strongly visible actions
\cite{K2007b}, 
one concludes from Theorem \ref{thm:bdd}
 that such examples for groups of type A
 are rare:
\begin{example}
[tensor product]
\label{ex:tensorA}
Suppose ${\mathfrak{g}}_{\mathbb{C}}={\mathfrak{sl}}(n,{\mathbb{C}})$.  
Then \eqref{eqn:bddt} holds 
 iff 
 $({\mathfrak{g}}_{\mathbb{C}}, {{\mathfrak{h}}_1}_{\mathbb{C}}, 
   {{\mathfrak{h}}_2}_{\mathbb{C}})$ is isomorphic
 to 
$
   ({\mathfrak{sl}}_2, {\mathfrak{so}}_{2}, {\mathfrak{so}}_{2})
$
or
$
   ({\mathfrak{sl}}_4, {\mathfrak{sp}}_{2}, {\mathfrak{sp}}_{2}).  
$
\end{example}

For groups of type BD, one has:

\begin{example}
[tensor product]
\label{ex:tensorOpq}
Let $G=O(p,q)$, 
 and $H_1$, $H_2$ be $O(p-1,q)$ or $O(p,q-1)$.  
Then \eqref{eqn:bddt} holds.  
In particular, 
 the tensor product $\Pi_{\delta,\lambda}^{p,q} \otimes \Pi_{\varepsilon,\nu}^{p,q}$
 decomposes into irreducible unitary representations
 with uniformly bounded multiplicities
 for any $\delta, \varepsilon \in \{+,-\}$, 
 $\lambda \in A_{\delta}(p,q)$, 
 $\nu \in A_{\varepsilon}(p,q)$.  
\end{example}

\begin{example}
[tensor product]
\label{ex:SO8}
Let $G=O(2p,2q)$ with $p+q=4$, 
 $H=O(2p-1,2q)$, 
 and $G'=U(p,q)$.  
Then \eqref{eqn:bddt} holds.  
\end{example}

Proofs of the assertions in Appendix
 will be given in another paper.

\end{document}